\documentclass[english,french,reqno,twoside,a4paper]{amsart}
\newcommand*{\letitre}{Algèbres différentielles de formes quasi-Jacobi d'indice nul}
\newcommand*{\lesmotscles}{Formes elliptiques, formes de Jacobi, déformations formelles, crochets de Rankin-Cohen}
\newcommand*{\lescodesMSC}{Primaire : 11F50, 16S80 Secondaire : 11F11, 11F25, 16W25, 53D55}
\newcommand*{\lesujet}{Math. Subj. Class~\lescodesMSC}
\newcommand{\leresume}{La notion de double profondeur attachée aux formes quasi-Jacobi permet de distinguer dans l'algèbre \(\AlgQJS\) des formes singulières quasi-Jacobi d'indice nul certaines sous-algèbres significatives (formes de type modulaire, formes de type elliptique, formes de Jacobi). On étudie la stabilité de ces sous-algèbres par les dérivations de \(\AlgQJS\) et par certaines suites d'opérateurs bidifférentiels constituant des analogues de crochets de Rankin-Cohen ou de transvectants.}
\newcommand{\leresumeanglais}{The notion of double depth associated with quasi-Jacobi forms allows distinguishing, within the algebra \(\AlgQJS\) of quasi-Jacobi singular forms of index zero, certain significant subalgebras (modular-type forms, elliptic-type forms, Jacobi forms). We study the stability of these subalgebras under the derivations of \(\AlgQJS\) and through certain sequences of bidifferential operators constituting analogs of Rankin-Cohen brackets or transvectants.
\begin{center}
\textbf{An english translation of this text is available at~\url{https://hal.science/hal-04735930}}. 
\end{center}}
\usepackage[utf8]{inputenc} 
\usepackage[T1]{fontenc}
\usepackage[narrowiints,partialup]{kpfonts}
\usepackage{xcolor}
\usepackage{babel}
\usepackage[all]{xy}
\usepackage[showonlyrefs]{mathtools}
\usepackage[useregional]{datetime2}
\usepackage[paper=a4paper]{geometry}
\usepackage{varioref,hyperref}
\usepackage{microtype}
\usepackage[all]{nowidow}
\usepackage{enumitem,placeins,tabularx,tikz,rotating}
\usetikzlibrary{calc}
\newcommand*\texten[1]{\foreignlanguage{english}{\textit{#1}}}

\makeatletter
 \def\@@and{et}
 \def\author@andify{%
	\nxandlist {\unskip ,\penalty-1 \space\ignorespaces}%
    	{\unskip {} \@@and~}%
    	{\unskip \penalty-2 \space \@@and~}%
 }
\makeatother
\makeatletter
 \newlength{\h@uteurnumerateur}
 \newlength{\h@uteurdenominateur}
\newcommand*{\qd}[2]{
  \mathchoice%
  {
    \settoheight{\h@uteurnumerateur}{\ensuremath{\displaystyle{#1#2}}}%
    \settoheight{\h@uteurdenominateur}{\ensuremath{\displaystyle{#1#2}}}%
    \raisebox{0.5\h@uteurnumerateur}{\ensuremath{\displaystyle{#1}}}%
    \mkern-5mu\diagup\mkern-4mu%
    \raisebox{-0.5\h@uteurdenominateur}{\ensuremath{\displaystyle{#2}}}%
  }
  {
    \settoheight{\h@uteurnumerateur}{\ensuremath{\textstyle{#1#2}}}%
    \settoheight{\h@uteurdenominateur}{\ensuremath{\textstyle{#1#2}}}%
  \raisebox{0.2\h@uteurnumerateur}{\ensuremath{\textstyle{#1}}}%
  /%
  \raisebox{-0.2\h@uteurdenominateur}{\ensuremath{\textstyle{#2}}}%
  }
  {
    \settoheight{\h@uteurnumerateur}{\ensuremath{\scriptstyle{#1#2}}}%
    \settoheight{\h@uteurdenominateur}{\ensuremath{\scriptstyle{#1#2}}}%
  \raisebox{0.2\h@uteurnumerateur}{\ensuremath{\scriptstyle{#1}}}%
  /%
  \raisebox{-0.2\h@uteurdenominateur}{\ensuremath{\scriptstyle{#2}}}%
  }
  {
    \settoheight{\h@uteurnumerateur}{\ensuremath{\scriptscriptstyle{#1#2}}}%
    \settoheight{\h@uteurdenominateur}{\ensuremath{\scriptscriptstyle{#1#2}}}%
  \raisebox{0.2\h@uteurnumerateur}{\ensuremath{\scriptscriptstyle{#1}}}%
  /%
  \raisebox{-0.2\h@uteurdenominateur}{\ensuremath{\scriptscriptstyle{#2}}}%
  }
}
\hypersetup{%
   unicode=true,
   hidelinks,%
   pdftitle={\letitre},%
   pdfauthor={\textcopyright\ François Dumas, François Martin et Emmanuel Royer},%
   pdfsubject={\lesujet},%
   pdfkeywords={\lesmotscles},%
   baseurl={https://royer.perso.math.cnrs.fr},%
   pdflang=fr-FR,
  }
\newcommand*{\AjoutRouge}[1]{%
#1
}
\newcommand*{\Ajout}[1]{%
#1
}
\makeatletter
\def\mathcolor#1#{\@mathcolor{#1}}
\def\@mathcolor#1#2#3{%
  \protect\leavevmode
  \begingroup
    \color#1{#2}#3%
  \endgroup
}
\makeatother
\newcommand*{\Ajoutmath}[1]{\mathcolor{magenta}{#1}}
\DeclarePairedDelimiter{\abs}{\lvert}{\rvert}
\DeclarePairedDelimiterXPP\accol[3]{}\{\}{_{#3}}{
	\ifblank{#1}{\phantom{f}}{#1},
	\ifblank{#2}{\phantom{f}}{#2}
	}
\DeclareMathOperator{\AC}{H}
\newcommand*{\AlgJS}{\mathrm{JS}}
\newcommand*{\AlgM}{\mathrm{M}}
\newcommand*{\AlgQJS}{\mathrm{JS}^{\infty}}
\newcommand*{\AlgQJell}{\mathrm{JS}^{0,\infty}}
\newcommand*{\AlgQJmod}{\mathrm{JS}^{\infty,0}}
\newcommand*{\AlgQM}{\mathrm{M}^{\infty}}
\newcommand*{\CC}{\mathbb{C}}
\DeclareMathOperator{\CM}{CM}
\DeclareMathOperator{\cQ}{Q}
\DeclarePairedDelimiterXPP\Crochet[3]{}\llbracket\rrbracket{_{#3}}{
	\ifblank{#1}{\phantom{f}}{#1},
	\ifblank{#2}{\phantom{f}}{#2}
	}
\DeclarePairedDelimiterXPP\crochet[3]{}\lbrack\rbrack{_{#3}}{
	\ifblank{#1}{\phantom{f}}{#1},
	\ifblank{#2}{\phantom{f}}{#2}
}
\DeclarePairedDelimiter{\projZ}{\lVert}{\rVert}
\DeclareMathOperator{\derpoidsuncolored}{\delta}
\newcommand*{\derpoids}{\Ajoutmath{\derpoidsuncolored}}
\newcommand*{\dtau}{\mathop{\partial}\nolimits_{\tau}}
\newcommand*{\dz}{\mathop{\partial}\nolimits_{z}}
\DeclareMathOperator{\fe}{e}
\DeclareMathOperator{\eisen}{e}
\DeclareMathOperator{\Eisen}{E}
\DeclarePairedDelimiter{\Ent}{\lfloor}{\rfloor}
\DeclareMathOperator{\fd}{d}
\newcommand*{\fds}{\fd_{\mathrm{S}}}
\newcommand*{\fdmod}{\fd^{\infty,0}_{\mathrm{S}}}
\newcommand*{\fdell}{\fd^{0,\infty}_{\mathrm{S}}}
\newcommand*{\fdq}{\fd^{\infty}_{\mathrm{S}}}
\newcommand*{\fonc}{{\CC^\pkC}}
\DeclareMathOperator{\ia}{I}
\newcommand*{\ic}{\mathrm{i}}
\DeclareMathOperator{\J}{J}
\newcommand*{\Jac}{{\sldz\ltimes\Z^2}}
\newcommand*{\jc}{\mathrm{j}}
\newcommand*{\JS}[1]{\mathrm{JS}_{#1}}
\newcommand*{\JSell}[1]{\mathrm{JS}^{0,\infty}_{#1}}
\newcommand*{\JSmod}[1]{\mathrm{JS}^{\infty,0}_{#1}}
\newcommand*{\N}{\mathbb{N}}
\newcommand{\Ob}{\mathop{\mathrm{Ob}}\nolimits^*}
\DeclareMathOperator{\ObCDMR}{\mathrm{Ob}}
\DeclarePairedDelimiterXPP\parent[3]{}\lparen\rparen{_{#3}}{
	\ifblank{#1}{\phantom{f}}{#1},
	\ifblank{#2}{\phantom{f}}{#2}
	}
\DeclareMathOperator{\olddeltauncolored}{\eta}
\newcommand*{\olddelta}{\Ajoutmath{\olddeltauncolored}}
\DeclareMathOperator{\pa}{P}
\newcommand*{\pk}{\mathcal{H}}
\newcommand*{\pkC}{{\pk\times\CC}}
\newcommand*{\planck}{\hbar}
\newcommand*{\QJS}[3]{\mathrm{JS}_{#1}^{\leq#2,#3}}
\newcommand*{\QJSpoids}[1]{\mathrm{JS}_{#1}^{\infty}}
\newcommand{\Reseau}{\mathcal{R}}
\newcommand*{\Sing}{\mathcal{S}}
\newcommand*{\SL}{\mathrm{SL}}
\newcommand*{\sldz}{\SL(2,\Z)}
\DeclareMathOperator{\wpdp}{W}
\DeclareMathOperator{\X}{X}
\DeclareMathOperator{\Y}{Y}
\newcommand*{\Z}{\mathbb{Z}}
\makeatletter
\def\l@paragraph{\@tocline{4}{0pt}{5pc}{7pc}{}}
\@xp\gdef\csname r@tocindent4\endcsname{0pt}
\def\paragraph{\@startsection{paragraph}{4}%
  \z@{.5\linespacing\@plus.7\linespacing}{-.5em}%
  {\normalfont\itshape}}
\makeatother
\theoremstyle{plain}
\newtheorem{thm}{Théorème}
\newtheorem{lem}[thm]{Lemme}
\newtheorem{prop}[thm]{Proposition}
\newtheorem{cor}[thm]{Corollaire}
\theoremstyle{definition}
\newtheorem{dfn}[thm]{Définition}
\theoremstyle{remark}
\newtheorem{rem}[thm]{Remarque}
\newcommand\tikzmark[2]{%
\tikz[remember picture,baseline] \node[inner sep=2pt,outer sep=0] (#1){#2};%
}
\newcommand\link[2]{%
\begin{tikzpicture}[remember picture, overlay, >=stealth, shift={(0,0)}]
  \draw[->] (#1) to (#2);
\end{tikzpicture}%
}
\newcommand*{\rose}[1]{{#1}}
\newcommand*{\bleu}[1]{{#1}}
\newcounter{cellindex}

\newcolumntype{C}{>{\stepcounter{cellindex}%
\begin{tikzpicture}[remember picture,baseline=(Cell-\thecellindex-left.base),inner sep=0pt]
\node (Cell-\thecellindex-left){\strut};
\end{tikzpicture}
}c<{\begin{tikzpicture}[remember picture,baseline=(Cell-\thecellindex-right.base),inner sep=0pt]
\node (Cell-\thecellindex-right){\strut};
\end{tikzpicture}}} 

\newcommand{\MeasureLastTable}[1]{
\pgfmathtruncatemacro{\LeftCellIndex}{\thecellindex-1}
\pgfmathtruncatemacro{\AboveCellIndex}{\thecellindex-#1}
\begin{tikzpicture}[overlay,remember picture]
    \path let \p0 = (Cell-\thecellindex-right.east), 
    \p1 = (Cell-\thecellindex-left.west),
    \p2 = (Cell-\LeftCellIndex-right.east),
    \p3 = (Cell-\LeftCellIndex-left.west) in 
    \pgfextra{\pgfmathsetmacro{\tmp}{(\x0+\x1-\x2-\x3)/4}
    \xdef\HorSep{\tmp pt} 
    };
    \path let \p0 = (Cell-\thecellindex-right.north west), 
    \p1 = (Cell-\thecellindex-left.south west),
    \p2 = (Cell-\AboveCellIndex-right.north west),
    \p3 = (Cell-\AboveCellIndex-left.south west) in 
    \pgfextra{\pgfmathsetmacro{\tmp}{(\y2+\y3-\y0-\y1)/4}
    \xdef\VertSep{\tmp pt} 
    }; 
\end{tikzpicture}%
}

\newcommand{\CrossOut}[2][]{
\tikz[overlay,remember picture]{
\path let \p0 = (Cell-#2-left.west), 
    \p1 = (Cell-#2-right.east),
    \p2 = (Cell-#2-left.north west),
    \p3 = (Cell-#2-left.south west) in
    \pgfextra{\pgfmathsetmacro{\tmp}{\HorSep-(\x1-\x0)/2}
    \xdef\myxoffset{\tmp pt}
    \pgfmathsetmacro{\tmp}{\VertSep-(\y2-\y3)/2}
    \xdef\myyoffset{\tmp pt}
    };
\draw[thick,#1] ([xshift=-\myxoffset,yshift=-\myyoffset]Cell-#2-left.south west) -- 
([xshift=\myxoffset,yshift=\myyoffset]Cell-#2-right.north east);
}}
\makeatletter
\newenvironment{abstracts}{%
  \ifx\maketitle\relax
    \ClassWarning{\@classname}{Abstract should precede
      \protect\maketitle\space in AMS document classes; reported}%
  \fi
  \global\setbox\abstractbox=\vtop \bgroup
    \normalfont\Small
    \list{}{\labelwidth\z@
      \leftmargin3pc \rightmargin\leftmargin
      \listparindent\normalparindent \itemindent\z@
      \parsep\z@ \@plus\p@
      
      \itemsep\medskipamount
    }%
}{%
  \endlist\egroup
  \ifx\@setabstract\relax \@setabstracta \fi
}

\newcommand{\abstractin}[1]{%
  \otherlanguage{#1}%
  \item[\hskip\labelsep\scshape\abstractname.]%
}
\makeatother
\begin{document}
\title{\letitre}
\author[F. Dumas]{Fran{ç}ois Dumas}
\address{%
Fran{ç}ois Dumas\\
Université Clermont Auvergne -- CNRS\\
Laboratoire de mathématiques Blaise Pascal -- UMR6620\\
F-63000 Clermont-Ferrand\\
France
}
\email{{francois.dumas@uca.fr}}
\author[F. Martin]{Fran{ç}ois Martin}
\address{%
Fran{ç}ois Martin\\
Université Clermont Auvergne -- CNRS\\
Laboratoire de mathématiques Blaise Pascal -- UMR6620\\
F-63000 Clermont-Ferrand\\
France
}
\email{{francois.martin@uca.fr}}
\author[E. Royer]{Emmanuel Royer}
\address{%
Emmanuel Royer\\
Université Clermont Auvergne -- CNRS -- Wolfgang Pauli Institute\\
Institut CNRS Pauli -- IRL2842 \\
A-1090 Wien\\
Autriche
}
\address{%
Emmanuel Royer\\
CNRS – Université de Montréal CRM – CNRS\\
IRL3457\\
Montréal\\
Canada
}
\email{{emmanuel.royer@math.cnrs.fr}}
\date{\DTMnow}
\thanks{Les travaux des deux derniers auteurs sont partiellement financés par le projet ANR-23-CE40-0006-01 Gaec. Le troisième auteur a bénéficié de précieux échanges avec Yuk-kam Lau et Ben Kane lors de visites à l'\emph{\texten{Institute of Mathematical Research}} de \emph{Hong Kong University} permises par le projet Hubert Curien Procore \no~48166WK. Il remercie Timothy Browning de l'\emph{Institute of Science and Technology Austria} pour son accueil.  Dans le but d'une publication en libre accès, cette version du texte est distribuée sous licence \href{https://creativecommons.org/licenses/by/4.0/deed.fr}{CC-BY}}
\subjclass[2020]{\lescodesMSC}
\begin{abstracts}
\abstractin{french}
\leresume

\abstractin{english}
\leresumeanglais
\end{abstracts}
\maketitle
\tableofcontents
\newcommand{\lintro}{Introduction}
\section{\lintro}
Cet article propose une étude analytique et algébrique des formes singulières quasi-Jacobi d'indice nul. Il étudie en particulier la stabilité par dérivations de certaines sous-algèbres significatives (formes elliptiques, formes quasi-Jacobi de type quasielliptique, formes quasi-Jacobi de type quasimodulaire), avec l'objectif de construire des suites d'opérateurs bidifférentiels constituant des déformations formelles de ces algèbres, crochets de Rankin-Cohen ou transvectants.

Pour les actions (paramétrées par un entier positif, le poids) du groupe modulaire \(\sldz\) sur l'algèbre des fonctions d'une variable complexe \(\tau\) du demi-plan de Poincaré \(\pk\) à valeurs dans \(\CC\), il est bien connu  que l'algèbre \(\AlgM\) des formes modulaires (graduée par le poids) n'est pas stable par la dérivation \(\partial_\tau\). Il y a au moins deux voies pour lever cette obstruction. La première est de construire de façon canonique une suite d'opérateurs bidifférentiels en \(\partial_\tau\) dits crochets de Rankin-Cohen qui stabilisent M (cf. \cite{zbMATH05808162}) et qui constituent de plus (cf. \cite{zbMATH05156388}, \cite{zbMATH07362171} et \cite{zbMATH05808162}) une déformation formelle de l'algèbre \(\AlgM\) (au sens de \cite[\texten{Chapter}~13]{zbMATH06054532}). La seconde est de définir au-dessus de \(\AlgM\) l'algèbre \(\AlgQM\) des formes quasimodulaires qui est par construction stable par \(\partial_\tau\), graduée par le poids et filtrée par la profondeur (cf. \cite{zbMATH05808162}, \cite{zbMATH06128504}). 
Ces deux points de vue sont étroitement liés puisqu'une méthode pour montrer la stabilité de \(\AlgM\) par les crochets de Rankin-Cohen passe par un prolongement de leur définition à l'algèbre \(\AlgQM\) (voir \cite[\texten{Section}~5]{MR1280058}, ou  \cite[\texten{Proposition}~9]{zbMATH07362171}). C'est une démarche comparable qui est proposée dans cet article pour l'action du groupe de Jacobi sur des fonctions en deux variables. Elle nécessite de reprendre dans un contexte formalisé et unifié diverses notions dispersées dans la littérature sur les formes de Jacobi et formes quasi-Jacobi (voir par exemple \cite{MR4281261}, \cite{zbMATH05953688}, \cite{MR0781735}, \cite{Fogliasso}).

Dans ce qui suit on considère les actions (paramétrées par deux entiers positifs, le poids et l'indice) du groupe de Jacobi \(\Jac\) sur les fonctions de deux variables complexes \((\tau,z)\) de \(\pkC\) dans \(\CC\). La notion de forme de Jacobi singulière s'en déduit (définition \ref{eq_dfnJacobiSing}), le terme singulier faisant ici référence aux hypothèses analytiques de périodicité et de méromorphie nécessaires que nous explicitons pour plus de clarté à la définition \ref{defsing}. En notant \(\AlgJS_k\) l'espace vectoriel des formes de Jacobi singulières d'indice nul et de poids \(k\), le théorème \ref{prop_strucJS} décrit l'algèbre graduée \(\AlgJS=\bigoplus\AlgJS_k\) comme l'algèbre des polyn\^omes \(\CC[\wp,\partial_z\wp, \eisen_4]\) où \(\wp\) est la fonction de Weierstra\ss{} et \(\eisen_4\) la série d'Eisenstein de poids 4. Elle co\"incide donc avec l'algèbre des formes elliptiques au sens de la définition \ref{dfn_forfonell}.  La fin de la première section de l'article est consacrée à la détermination (proposition \ref{prop_dimJS}) de la dimension des sous-espaces \(\AlgJS_k\).

L'algèbre \(\AlgJS\), comme sa sous-algèbre \(\AlgM\), n'est pas stable par la dérivation \(\partial_\tau\). Ceci conduit à introduire dans la section~\ref{sec_dectiondeux} la notion de forme quasi-Jacobi singulière d'indice nul, à laquelle sont attachés par construction un poids \(k\in\N\) et une double profondeur \((s_1,s_2)\in\N^2\)  (voir définition \ref{dfn_JS}). 
Ces formes quasi-Jacobi singulières se structurent en une algèbre \(\AlgQJS\) graduée par le poids et doublement filtrée par la profondeur, que le théorème \ref{thm_strucQJ} décrit comme l'algèbre de polyn\^omes en cinq variables \(\AlgQJS=\CC[\wp,\partial_z\wp, \eisen_4,\eisen_2,\Eisen_1]\) où \(\eisen_2\) est la série d'Eisenstein de poids~\(2\) et profondeur \((1,0)\) et \(\Eisen_1\) la première fonction d'Eisenstein décalée de \Ajout{poids} \(1\) et profondeur \((0,1)\). Les deux sous-algèbres \({\AlgQJmod}=\CC[\wp,\partial_z\wp, \eisen_4,\eisen_2]\) et \({\AlgQJell}=\CC[\wp,\partial_z\wp, \eisen_4,\Eisen_1]\) intermédiaires entre \(\AlgJS\) et \(\AlgQJS\)  correspondent aux formes quasi-Jacobi de profondeur \((s_1,0)\) et \((0,s_2)\), respectivement dénommées de type quasimodulaire et de type quasielliptique. 

La section~\ref{sec_trois} de l'article est consacrée à la construction de déformations formelles sur chacune des quatre algèbres en jeu, et leurs liens avec les crochets de Rankin-Cohen classiques sur la sous-algèbre \(\AlgM\). La dérivation \(\partial_\tau\) de \(\AlgQJS\) étant homogène de degré \(2\) pour le poids, on peut introduire à la proposition \ref{RCtauQ} des crochets de Rankin-Cohen sur \(\AlgQJS\) qui constituent une déformation formelle de \(\AlgQJS\) (voir \cite{zbMATH05156388} et \cite{zbMATH07362171}, suivant le principe initié en \cite{MR1280058}). En utilisant les arguments algébriques généraux de \cite[\texten{Theorem}~6]{zbMATH07362171}, on démontre au théorème \ref{RCtauA} que la sous-algèbre \({\AlgQJell}\) est stable par ces crochets qui prolongent ceux classiquement définis sur \(\AlgM\). La m\^eme méthode permet d'obtenir  au théorème \ref{RCtauE} une déformation formelle de l'algèbre \(\AlgJS\) des formes elliptiques prolongeant les crochets de Rankin-Cohen sur \(\AlgM\)  en considérant cette fois des opérateurs bidifférentiels en la dérivation \(d=\partial_\tau+\frac14\Eisen_1\partial_z\) (voir aussi avec une preuve différente \cite[\texten{Proposition}~2.15]{zbMATH05953688}). Pour le cas des formes quasi-Jacobi de type quasimodulaire, c'est par une stratégie très différente basée sur la notion de transvectants issue de la théorie classique des invariants (voir \cite{zbMATH01516969})  que l'on obtient au théorème \ref{TransB} une déformation formelle de l'algèbre \(\AlgQJmod\).

\section{Formes de Jacobi singulières}\label{sec_dectiondeux}
\subsection{Fonctions elliptiques associées à un réseau}\label{sub_ellres}
\cite[\texten{Chapter}~V]{zbMATH05500775} 
Soit \(\Reseau\) un réseau de \(\CC\). Une fonction \(f\colon\CC\to\CC\) méromorphe est dite \emph{elliptique pour \(\Reseau\)} si
\[
\forall z\in\CC\enspace\forall w\in \Reseau \quad f(z+w)=f(z). 
\] 
Un exemple fondamental de telle fonction est la fonction de Weierstra\ss{} associée au réseau \(\Reseau\) définie par
\[
\forall z\in\CC-\Reseau\quad\wp_{\Reseau}(z)=\frac{1}{z^2}+\sum_{w\in \Reseau-\{0\}}\left(\frac{1}{(z-w)^2}-\frac{1}{w^2}\right). 
\]
Toute fonction elliptique paire est une fraction rationnelle à coefficients complexes en \(\wp_\Reseau\) \cite[\texten{Proposition}~V.3.2]{zbMATH05500775}.
Pour tout entier \(k\geq 4\) pair, on définit le nombre complexe
\[
\eisen_{k,\Reseau}=\sum_{w\in \Reseau-\{0\}}w^{-k}. 
\]
La fonction \(\wp_\Reseau\) vérifie l'équation différentielle 
\begin{equation}\label{eq_derivrho}
\left(\wp_{\Reseau}'\right)^2=\wpdp_\Reseau(\wp_\Reseau) \quad\text{avec}\quad \wpdp_\Reseau(X)=4(X^3-15\eisen_{4,\Reseau}X-35\eisen_{6,\Reseau})\in\CC[X]. 
\end{equation}
Si \(P_1\), \(Q_1\), \(P_2\) et \(Q_2\) sont des fractions rationnelles, on a alors
\[
\left(P_1(\wp_\Reseau)+Q_1(\wp_\Reseau)\wp_{\Reseau}'\right)\left(P_2(\wp_\Reseau)+Q_2(\wp_\Reseau)\wp_{\Reseau}'\right)=\left(P_3(\wp_\Reseau)+Q_3(\wp_\Reseau)\wp_{\Reseau}'\right)
\]
avec 
\[
P_3=P_1P_2+\wpdp_\Reseau Q_1Q_{2}\quad\text{et}\quad Q_3=P_1Q_2+Q_1P_2. 
\]

En particulier, si \(P\) et \(Q\) sont deux fractions rationnelles de \(\CC(X)\) et si
\[
\widetilde{P}=\frac{P}{P^2-\wpdp_\Reseau Q^2} \quad\text{et}\quad\widetilde{Q}=-\frac{Q}{P^2-\wpdp_\Reseau Q^2}. 
\]
alors
\[
\left(P(\wp_\Reseau)+Q(\wp_\Reseau)\wp_{\Reseau}'\right)\left(\widetilde{P}(\wp_\Reseau)+\widetilde{Q}(\wp_\Reseau)\wp_{\Reseau}'\right)=1. 
\]
Ainsi, l'ensemble
\[
\mathcal{E}(\Reseau)=\CC\left(\wp_\Reseau\right)\oplus\CC\left(\wp_\Reseau\right)\wp_{\Reseau}'
\]
est un corps. Puisque \(\CC\left(\wp_\Reseau\right)\) est le corps des fonctions elliptiques paires, et puisque si \(f\) est elliptique et impaire, alors le quotient \(f/\wp_{\Reseau}'\) est elliptique et paire, on conclut que le corps \(\mathcal{E}(\Reseau)\) est l'ensemble des fonctions elliptiques pour \(\Reseau\).
\subsection{Formes elliptiques}
Pour tout \(\lambda\in\CC^*\), on a
\begin{align}
\wp_{\lambda \Reseau}(z) &= \lambda^{-2}\wp_{\Reseau}\left(\lambda^{-1}z\right) \label{eq_trsfreseauW}\\
\eisen_{k,\lambda \Reseau} &=\lambda^{-k}\eisen_{k,\Reseau}\label{eq_trsfreseaue}
\end{align}
de sorte que l'on peut se restreindre aux représentants des classes d'équivalences de réseaux par homothétie complexe. Tout réseau ayant une base \((w_1,w_2)\) avec \(w_2/w_1\) appartenant au demi-plan de Poincaré \(\pk\) des complexes de partie imaginaire strictement positive, on se restreint aux réseaux \(\Z\oplus\tau\Z\) avec \(\tau\in\pk\).

On définit alors pour tout entier \(k\geq 4\) pair la fonction d'Eisenstein de poids \(k\) par
\[
\begin{array}{ccccc}
\eisen_k & \colon & \pk & \to & \CC\\
& & \tau & \mapsto & \eisen_{k,\Z\oplus\tau\Z}. 
\end{array}
\]
C'est une forme modulaire de poids \(k\) sur \(\sldz\) dont le développement de Fourier\footnote{L'entier \(\sigma_{k-1}(n)\) est \(\sum_{d\mid n}d^{k-1}\). La suite \((B_n)_{n\geq 0}\) est définie par la série génératrice : \[\frac{t}{e^t-1}=\sum_{n=0}^{+\infty}B_n\frac{t^n}{n!}.\]} est
\begin{equation}\label{eq_devFEisen}
\eisen_k(\tau)=\frac{2^k\lvert B_{k}\rvert}{k!}\pi^k\left(1-\frac{2k}{B_{k}}\sum_{n=1}^{+\infty}\sigma_{k-1}(n)e^{2i\pi n\tau}\right) \quad(\text{\(k\geq 4\) pair}).
\end{equation}
On définit \(\eisen_2\) en étendant cette égalité à \(k=2\). La fonction \(\eisen_2\) n'est pas une forme modulaire.

De la même façon on définit la fonction de Weierstra\ss{} par
\[
\begin{array}{ccccc}
\wp & \colon & \pkC & \to & \CC\\
& & (\tau,z) & \mapsto & \wp_{\Z\oplus\tau\Z}(z). 
\end{array}
\]
\begin{dfn}\label{dfn_forfonell}
On appelle \emph{forme elliptique} tout élément de l'anneau \(\CC[\wp,\dz\wp,\eisen_4]\) et \emph{fonction elliptique} tout élément du corps de fractions \(\CC(\wp,\dz\wp,\eisen_4)\). 
\end{dfn}

\Ajout{
\begin{rem}\label{rem_conventionunedeux}
On notera aussi \(\eisen_k\) la fonction de deux variables induite par la fonction d'Eisenstein de poids \(k\) et on l'appellera encore fonction d'Eisenstein de poids \(k\) :
\begin{equation}\label{eq_conventionunedeux}
 \begin{array}{ccccc}
 \eisen_k & \colon &  \pkC & \to & \CC\\
& &  (\tau, z) & \mapsto & \eisen_k(\tau)
\end{array}
\qquad\text{(\(k \geq 2\) pair). }
\end{equation}
\end{rem}
\begin{rem}
Notre utilisation du terme \emph{elliptique} fait référence à la théorie des fonctions elliptiques associées à un réseau fixe, décrites à l'aide de la fonction de Weierstra\ss{} associée au réseau et de sa dérivée. La volonté de \og~libérer~\fg{} le réseau impose que nous ajoutions \(\eisen_4\) et \(\eisen_6\). L'équation différentielle 
\begin{equation}\label{eq_differfonda}
(\dz\wp)^2-4\wp^3+60\eisen_4\wp+140\eisen_6=0
\end{equation}
impose ensuite de retirer la fonction \(\eisen_6\) des générateurs : elle est en effet polynomiale en les fonctions algébriquement indépendantes \(\wp\), \(\dz\wp\) et \(\eisen_4\) (voir le théorème~\ref{prop_strucJS}).  On redémontrera~\eqref{eq_differfonda} indépendamment plus loin dans ce texte (voir l'équation~\eqref{eq_diffwp}). 
\end{rem}
}


\subsection{Formes de Jacobi singulières d'indice nul}
\subsubsection{Action du groupe de Jacobi\texorpdfstring{ sur \(\pkC\) et sur \(\fonc\)}{}}

Le groupe multiplicatif \(G=\sldz\) agit à droite sur le groupe additif \(H=\Z^2\) par
\begin{equation*}\label{Gsur H}
\forall g=\begin{pmatrix}a & b\\ c & d\end{pmatrix}\in G, \ \forall \Lambda=(\lambda,\mu)\in H, \ \Lambda g=(\lambda a+\mu c,\lambda b+\mu d)
\end{equation*}
Le groupe de Jacobi est le produit semi-direct \(G\ltimes H=\Jac\) qui s'en déduit, avec le produit 
\begin{equation*}
\forall g,g'\in G, \ \forall \Lambda,\Lambda'\in H, \  (g,\Lambda)(g',\Lambda')=(gg', \Lambda g'+\Lambda').
\end{equation*}
Les groupes \(G\) et \(H\) agissent à gauche sur \(\pkC\)
\begin{align*}
\forall g=\begin{pmatrix}a & b\\ c & d\end{pmatrix}\in G, \ \forall(\tau,z)\in\pkC, \ g(\tau,z)&=\left(\frac{a\tau+b}{c\tau+d},\frac{z}{c\tau+d}\right),\\
\forall \Lambda=(\lambda,\mu)\in H, \ \forall(\tau,z)\in\pkC, \ \Lambda(\tau,z)&=(\tau,z+\lambda\tau+\mu).
\end{align*}
On en déduit une action à droite \(\vert_G\) de \(G=\sldz\) et une action à droite \(\vert_H\) de \(H=\Z^2\) sur l'algèbre de fonctions \(\fonc\) définies par
\begin{align}\forall g=\begin{pmatrix}a & b\\ c & d\end{pmatrix}\in G, \ \forall f\in\fonc,\ &f\vert_G g : (\tau,z)\mapsto f\left(\frac{a\tau+b}{c\tau+d},\frac{z}{c\tau+d}\right)\label{actGsurFonc}\\
\forall \Lambda=(\lambda,\mu)\in H, \ \forall f\in\fonc,\ &f\vert_H\Lambda : (\tau,z)\mapsto f(\tau,z+\lambda\tau+\mu).\label{actHsurFonc}
\end{align}
Ces deux actions sont compatibles au sens où
\begin{equation*}
\forall g\in G,\ \forall\Lambda\in H,\ \forall f\in\fonc, \ (f\vert_G g)\vert_H \Lambda g=(f\vert_H \Lambda)\vert_G g.
\end{equation*}
Ceci permet d'en déduire une action  à droite du groupe de Jacobi \(\Jac\) sur l'algèbre de fonctions \(\fonc\) :
\begin{equation}\label{actJsurFonc}
\forall g\in G,\ \forall\Lambda\in H,\ \forall f\in\fonc, \ f\vert_{G\ltimes H}(g,\Lambda)=(f\vert_G g)\vert_H \Lambda.
\end{equation}
En d'autres termes, pour tout \((\tau,z)\in\pkC\)
\begin{equation}
f\vert_{G\ltimes H}\left(\begin{pmatrix}a & b\\ c & d\end{pmatrix},(\lambda,\mu)\right)(\tau,z)=f\left(\frac{a\tau+b}{c\tau+d},\frac{z+\lambda\tau+\mu}{c\tau+d}\right). 
\end{equation}
Plus généralement si \(\nu\) est une application de \(\Jac\) dans \(\fonc\), alors l'application 
\begin{equation}\label{actJsurFoncnu}
(f,(g,\Lambda))\mapsto\nu(g,\Lambda)\left(f\vert_{G\ltimes H}(g,\Lambda)\right)
\end{equation}
définit une action à droite du groupe de Jacobi sur l'algèbre \(\fonc\) si et seulement si \(\nu\) est un 1-cocycle pour l'action~\eqref{actJsurFonc}, c'est-à-dire vérifie
\begin{equation}\label{def1cocycle}
\nu\left((g,\Lambda)(g',\Lambda')\right)=\left(\nu(g,\Lambda)\vert_{G\ltimes H}(g',\Lambda')\right)\nu(g',\Lambda').
\end{equation}
Un tel 1-cocycle \(\nu\) peut être obtenu à partir d'un 1-cocycle \(\nu_G\) pour l'action~\eqref{actGsurFonc} de \(G\) et d'un 1-cocycle \(\nu_H\) pour l'action~\eqref{actHsurFonc} de \(H\) en posant
\begin{equation}\label{cocycleGH}\forall(g,\Lambda)\in G\ltimes H, \ \nu(g,\Lambda)=(\nu_G(g)\vert_H\Lambda)\nu_H(\Lambda)
\end{equation}
qui vérifie la relation~\eqref{def1cocycle} si et seulement si on a la condition de compatibilité
\begin{equation}\label{compatible1cocycle}\forall(g,\Lambda)\in G\ltimes H, \ \left(\nu_G(g)\vert_H\Lambda g\right)\nu_H(\Lambda g)=\nu_G(g)\left(\nu_H(\Lambda)\vert_G g\right).
\end{equation}

Soit \(j\colon\sldz\to\fonc\) et \(\ell\colon\sldz\to\fonc\) définies pour \(g=\begin{psmallmatrix}a & b\\c & d\end{psmallmatrix}\) et \((\tau,z)\in\pkC\) par
\[
j(g)(\tau,z)=c\tau+d,\qquad\ell(g)(\tau,z)=\fe\left(-\frac{cz^2}{c\tau+d}\right)
\]
où \(\fe\colon\xi\mapsto\exp(2\ic\pi\xi)\). Ce sont des  \(1\)-cocycles de \(\sldz\) dans \(\fonc\).  Pour tous entiers positifs \(k\) et \(m\), l'application \(j^k\ell^m\) en est donc un aussi. 

Soit \(p\colon\Z^2\to\fonc\) définie pour \(\Lambda=(\lambda,\mu)\) et \((\tau,z)\in\pkC\) par
\[
p(\Lambda)(\tau,z)=\fe(\lambda^2\tau+2\lambda z). 
\]
C'est un \(1\)-cocycle de \(\Z^2\) dans \(\fonc\).  Pour tout entier positif \(m'\), l'application \(p^{m'}\) en est donc un aussi. 

Suivant la construction de~\eqref{cocycleGH}, considérons alors l'application
\[
\begin{array}{ccccc}
\nu_{k,m,m'}	& \colon	& \sldz\ltimes\Z^2	& \to		& \fonc\\
	& 		& (g,\Lambda) 		& \mapsto	& \left((j^k\ell^m)(g)\vert_{\Z^2}\Lambda\right) p^{m'}(\Lambda).
\end{array}
\]
La condition de compatibilité~\eqref{compatible1cocycle} est satisfaite si et seulement si \(m'=m\) et on en déduit que \(\nu_{k,m,m}=\nu_{k,m}\) est un \(1\)-cocycle pour l'action~\eqref{actJsurFoncnu} de \(\Jac\). 

Finalement, si \(k\) et \(m\) sont des entiers positifs, on définit une action de \(\Jac\) sur \(\fonc\) en posant
\begin{equation}\label{eq_vertkm}
\left(f\vert_{k,m}A\right)(\tau,z)=(c\tau+d)^{-k}\fe^{m}\left(-\frac{c(z+\lambda\tau+\mu)^2}{c\tau+d}+\lambda^2\tau+2\lambda z\right)f\left(\frac{a\tau+b}{c\tau+d},\frac{z+\lambda\tau+\mu}{c\tau+d}\right)
\end{equation}
pour tout \(A=(g,\Lambda)=\left(\begin{psmallmatrix}a & b\\c & d\end{psmallmatrix},(\lambda,\mu)\right) \in \Jac\), avec \(\fe^{m}(\xi)=\exp(2\ic m\pi\xi)\).
%

\subsubsection{Définition et exemples fondamentaux}\label{subsec_jacosingexple}

\begin{dfn}\label{defsing}
Une fonction \(f\colon\pkC\to\CC\) est \emph{singulière} si :
\begin{itemize}
\item pour tout \(\tau\in\pk\), la fonction \(z\mapsto f(\tau,z)\) est 1-périodique, méromorphe sur \(\CC\) et telle que ses seuls pôles sont les points du réseau \(\Z\oplus\tau\Z\), tous de m\^eme ordre, cet ordre étant de plus indépendant de \(\tau\) ;
\item la fonction \(\tau\mapsto f(\tau,z)\) est 1-périodique ;
\item les coefficients de Laurent de \(z\mapsto f(\tau,z)\) en \(0\) sont des fonctions holomorphes sur \(\pk\) et en l'infini. 
\end{itemize}
On note \(\Sing\) l'ensemble des fonctions singulières. 
\end{dfn}

\begin{rem}
Explicitons la troisième condition : notons \(A_n\) la fonction \(n\)-ième coefficient de Laurent de \(z\mapsto f(\tau,z)\) en \(0\). Par la deuxième condition, les fonctions \(A_n\) sont 1-périodiques. On demande donc qu'elles soient holomorphes sur \(\pk\) et admettent un développement de Fourier de la forme
\[
A_n(\tau)=\sum_{r=0}^{+\infty}\widehat{A_n}(r)\fe(r\tau). 
\]
\end{rem}

\begin{dfn}\label{eq_dfnJacobiSing}
Soit \(k\) et \(m\) des entiers positifs. Une fonction singulière \(f\colon\pkC\to\CC\) est une \emph{forme de Jacobi singulière}\footnote{La définition de forme de Jacobi méromorphe ne semble pas fixée. On s'inspire ici de~\cite[\S~3.2]{zbMATH06346312}.} d'indice \(m\) et de poids \(k\) si elle vérifie \(f\vert_{k,m}A=f\) pour tout \(A\in \Jac \). 
\end{dfn}

De façon explicite, une fonction singulière est une forme de Jacobi singulière d'indice \(m\) et de poids \(k\) si et seulement si elle vérifie les deux relations suivantes :
\begin{itemize}
\item pour tout \((\lambda,\mu)\in\Z^2\),
\begin{equation}\label{eq_trsfellf}
f(\tau,z+\lambda\tau+\mu)=e^{-2i\pi m(\lambda^2\tau+2\lambda z)}f(\tau,z)\text{ ;}
\end{equation}
\item pour toute \(\bigl(\begin{smallmatrix}a & b\\ c & d\end{smallmatrix}\bigr)\in\sldz\), 
\begin{equation}\label{eq_trsfmodf}
f\left(\frac{a\tau+b}{c\tau+d},\frac{z}{c\tau+d}\right)=e^{2i\pi mcz^2/(c\tau+d)}(c\tau+d)^kf(\tau,z)\text{ ;}
\end{equation}
\end{itemize}

\AjoutRouge{
\begin{rem}
Après circulation d'une première version de ce travail, Jan-Willem M. van Ittersum nous a informé de son article~\cite{zbMATH07634708} dans lequel il introduit la notion de \emph{forme de Jacobi strictement méromorphe}. Nous proposons une autre présentation d'un cas particulier de la sienne et, en particulier nous insistons sur le contexte analytique en évitant la notion de forme méromorphe à plusieurs variables. 
\end{rem}
}

On fixe une matrice \(g=\bigl(\begin{smallmatrix}a & b\\ c & d\end{smallmatrix}\bigr)\) dans \(\sldz\) et \((\lambda, \mu) \in \Z^2\). On a
\[
\wp\left(\frac{a\tau+b}{c\tau+d},\frac{z}{c\tau+d}\right)=\wp_{\Z\oplus\frac{a\tau+b}{c\tau+d}\Z}\left(\frac{z}{c\tau+d}\right).
\]
Or, \(\Z\oplus\frac{a\tau+b}{c\tau+d}\Z=\frac{1}{c\tau+d}\left(\Z\oplus\tau\Z\right)\), l'égalité~\eqref{eq_trsfreseauW} implique alors
\[
\wp_{\Z\oplus\frac{a\tau+b}{c\tau+d}\Z}\left(\frac{z}{c\tau+d}\right)=(c\tau+d)^2\wp_{\Z\oplus\tau\Z}(z)
\]
c'est-à-dire
\begin{equation}\label{eq_trsfmodwp}
\wp\left(\frac{a\tau+b}{c\tau+d},\frac{z}{c\tau+d}\right)=(c\tau+d)^2\wp(\tau,z). 
\end{equation}
D'autre part, par définition des fonctions elliptiques
\begin{equation}\label{eq_trsfellwp}
\wp\left(\tau,z+\lambda\tau+\mu\right)=\wp(\tau,z). 
\end{equation}
Notons \(\dz=\partial/\partial z\). En dérivant \eqref{eq_trsfmodwp} et \eqref{eq_trsfellwp}, on trouve
\begin{equation}\label{eq_trsfmoddzwp}
\left(\dz\wp\right)\left(\frac{a\tau+b}{c\tau+d},\frac{z}{c\tau+d}\right)=(c\tau+d)^3\dz\wp(\tau,z). 
\end{equation}
et
\begin{equation}\label{eq_trsfelldzwp}
\left(\dz\wp\right)\left(\tau,z+\lambda\tau+\mu\right)=\dz\wp(\tau,z). 
\end{equation}

\Ajout{Avec la convention faite dans la remarque~\ref{rem_conventionunedeux},} la fonction \(\eisen_4\) satisfait l'équation 
\begin{equation}\label{ploufplouf}
\eisen_4\left(\frac{a\tau+b}{c\tau+d}, \frac{z}{c\tau+d}\right)=(c\tau+d)^4\eisen_4(\tau, z).
\end{equation}
Le développement de Laurent de \(\wp\) est donné par
\begin{equation}\label{eq_delLaurent}
\wp(\tau,z)=\frac{1}{z^2}+\sum_{n=1}^{+\infty}(2n+1)\eisen_{2n+2}(\tau)z^{2n}
\end{equation}
\cite[\texten{Proposition}~V.2.5]{zbMATH05500775}, ce qui montre que \(\wp\) (et donc \(\dz\wp\)) sont singulières. 

Les relations \eqref{eq_trsfmodwp} à \eqref{eq_delLaurent} montrent donc que \(\wp, \dz\wp\) et \(\eisen_4\) sont des formes de Jacobi singulières d'indice nul et de poids respectifs \(2, 3\) et \(4\). La suite de cette section a pour objectif %
\AjoutRouge{de clarifier le cadre analytique de la \texten{Proposition 2.8} de~\cite{zbMATH05953688}. C'est une nouvelle réponse à l'objectif semblable suivi par~\cite[\texten{Proposition 2.7}]{zbMATH07634708}.}

\begin{thm}\label{prop_strucJS}

\begin{enumerate}
\item Les fonctions \(\wp, \dz\wp\) et \(\eisen_4\) sont algébriquement indépendantes.
\item L'algèbre des formes elliptiques est graduée par le poids. On note \(\displaystyle\AlgJS=\CC[\wp,\dz\wp,\eisen_4]=\bigoplus_{k \in \N} \AlgJS_k\) où \(\AlgJS_k\) est l'ensemble des éléments \(\displaystyle\sum_{\substack{(a,b,c)\in\N^3\\ 2a+3b+4c=k}}\alpha(a, b, c)\wp^a\left(\dz\wp\right)^b\eisen_4^c\) avec \(\alpha(a, b, c) \in \CC\).
\item Pour tout \(k \geq 0\), \(\AlgJS_k\) est l'ensemble des formes de Jacobi singulières d'indice nul et de poids \(k\).
\end{enumerate}
\end{thm}

\begin{proof}
Montrons d'abord l'indépendance algébrique de \(\wp\), \(\dz\wp\) et \(\eisen_4\). On a, pour tout \(\tau \in \pk\), \(\dz\wp\left(\tau, \dfrac{\tau}2\right)=0\) d'après \cite[\texten{Lemma}~V.2.8]{zbMATH05500775}. Grâce à~\eqref{eq_differfonda}, on a entre les fonctions \(\eisen_4\), \(\eisen_6\) et \(\widetilde{\wp}\colon\tau\mapsto\wp(\tau,\tau/2)\) une relation de dépendance algébrique~:
\[
\widetilde{\wp}^3-15\eisen_4\widetilde{\wp}-35\eisen_6=0.
\]
 Puisque \(\eisen_4\) et \(\eisen_6\) sont algébriquement indépendantes on en déduit que les fonctions \(\eisen_4\) et \(\widetilde{\wp}\) sont algébriquement indépendantes. Supposons alors \(\wp\), \(\dz\wp\) et \(\eisen_4\) algébriquement dépendantes. Il existerait un entier \(N\geq 1\) et une suite non nulle de complexes \(\alpha_{k,\ell}^{(i)}\) tels que 
 \[
 \sum_{i=0}^Nf_i(\dz\wp)^i=0\quad\text{avec}\quad f_i=\sum_{k,\ell}\alpha_{k,\ell}^{(i)}\eisen_4^\ell\wp^k. 
 \]
En spécialisant cette égalité à \(z=\tau/2\), on montre que \(\tau \mapsto f_0\left(\tau, \dfrac{\tau}2\right)\) est nulle, d'où, par récurrence, tous les \(\tau \mapsto f_i\left(\tau, \dfrac{\tau}2\right)\) sont nuls. Par indépendance algébrique de \(\eisen_4\) et \(\widetilde{\wp}\) il en résulte que tous les \(\alpha_{k,\ell}^{(i)}\) sont nuls d'où une contradiction. Cela démontre le point (1). Le point (2) en découle gr\^ace à la définition \ref{dfn_forfonell}.

Montrons (3). Grâce à~\eqref{eq_trsfmodf} appliquée à \((a,b,c,d)=(-1,0,0,-1)\), toute forme de Jacobi singulière d'indice nul et de poids \(k\) est paire en la variable \(z\) si \(k\) est pair et impaire en la variable \(z\) si \(k\) est impair. 

Soit \(f\) une forme de Jacobi singulière d'indice nul et de poids \(k\). Pour tout \(\tau\), la fonction \(z\mapsto f(\tau,z)\) est une fonction elliptique associée au réseau \(\Z\oplus\tau\Z\) dont les pôles sont des points du réseau. Dans \(\qd{\CC}{\Z\oplus\tau\Z}\), cette fonction a donc au plus un pôle (qui peut être multiple) et celui-ci est en \(0\). 

\paragraph{Cas du poids pair}
Si \(k\) est pair, pour tout \(\tau\) il existe \(P_\tau\in\CC[X]\) tel que 
\[
f(\tau,z)=P_\tau\left(\wp(\tau,z)\right) 
\] 
et le degré \(n_0\) de \(P_{\tau}\) est la moitié de l'ordre du pôle de \(z\mapsto f(\tau,z)\) en \(0\) \cite[\texten{Proposition}~V.3.1]{zbMATH05500775}. Il est donc indépendant de \(\tau\) et il existe des fonctions \(a_0,\dotsc,a_{n_0}\) de \(\pk\) dans \(\CC\) telles que
\begin{equation}\label{eq_poljacpair}
f(\tau,z)=\sum_{j=0}^{n_0}a_j(\tau)\wp(\tau,z)^j.
\end{equation}
Compte-tenu de~\eqref{eq_trsfmodf} et \eqref{eq_trsfmodwp}, on a
\[
\sum_{j=0}^{n_0}(c\tau+d)^{2j}a_j\left(\frac{a\tau+b}{c\tau+d}\right)\wp(\tau,z)^j=\sum_{j=0}^{n_0}(c\tau+d)^{k}a_j(\tau)\wp(\tau,z)^j.
\]
La famille \((\wp_{{\Z\oplus\tau\Z}}^j)_{j\in\mathbb{N}}\) est linéairement indépendante. On en déduit que chaque \(a_j\) est une fonction faiblement modulaire\footnote{Au sens de~\cite{MR0498338}, c'est-à-dire méromorphe sur le demi-plan de Poincaré et vérifiant les relations de modularité.} de poids \(k-2j\).

Montrons que les \(a_j\) sont holomorphes sur \(\pk\) et en l'infini. L'égalité entre le développement de Laurent
\[
\sum_{n=0}^{+\infty}A_n(\tau)z^{2n-2n_0}
\]
de \(z \mapsto f(\tau, z)\) en 0 et l'égalité~\eqref{eq_poljacpair} conduit grâce à~\eqref{eq_delLaurent} à 
\[
\sum_{n=0}^{+\infty}A_n(\tau)z^{2n}=\sum_{j=0}^{n_0}a_j(\tau)\left(\sum_{r=0}^{+\infty}\epsilon_r(\tau)z^{2r}\right)^jz^{2n_0-2j}
\]
les fonctions \(\epsilon_r\) holomorphes sur \(\pk\) étant définies par \(\epsilon_0=1\), \(\epsilon_1=0\) et \(\epsilon_r=(2r-1)\eisen_{2r}\) si \(r\geq 2\). On en déduit
\[
A_r=a_{n_0-r}+\sum_{j=n_0-r+1}^{n_0}a_j\sum_{\alpha_1+\dotsm+\alpha_j=r+j-n_0}\epsilon_{\alpha_1}\dotsm\epsilon_{\alpha_j}.
\]
Par récurrence, on obtient que les fonctions \(a_j\) sont holomorphes sur \(\pk\) et en l'infini.

Finalement, les fonctions \(a_j\) sont des formes modulaires donc des éléments de \(\CC[\eisen_4,\eisen_6]\). Ainsi, une forme de Jacobi singulière d'indice nul et de poids pair est un élément de \(\CC[\eisen_4,\eisen_6,\wp]\subset\CC[\eisen_4,\wp,(\dz\wp)^2]\subset\CC[\eisen_4,\wp,\dz\wp]\).

\paragraph{Cas du poids impair} Si \(k\) est impair, alors \(f\dz\wp\) est une forme de Jacobi singulière d'indice nul et de poids pair \(k+3\). On en déduit que \(f\dz\wp\in\CC[\eisen_4,\wp,(\dz\wp)^2]\) puis qu'il existe des polynômes \(P\) et \(Q\) tels que
\[
f=\frac{1}{\dz\wp}P(\eisen_4,\wp)+Q(\eisen_4,\wp,(\dz\wp)^2)\dz\wp.
\]
Pour tout \(\tau \in \pk\), \(z \mapsto f(\tau, z)\) n'a pas de pôle en \(z=\tau/2\), donc \[P\left(\eisen_4,\widetilde{\wp}\right)=0.\]  Par indépendance algébrique de \(\eisen_4\) et \(\widetilde{\wp}\), le polynôme \(P\) est nul. Ainsi \(f \in \CC[\wp,\dz\wp,\eisen_4]\).
\end{proof}

\subsubsection{Dimension des espaces \texorpdfstring{\(\AlgJS_k\)}{de formes de Jacobi singulières}}
Pour tout entier \(k \geq 0\), une base de l'espace \(\AlgJS_k\) est 
\begin{equation}\label{eq_base}
\left\{\wp^a(\dz\wp)^b\eisen_4^c \colon (a,b,c)\in\N^3, 2a+3b+4c=k\right\}. 
\end{equation}
L'équation \(2a+3b+4c=k\) équivaut à l'équation \(4a+6b=2k-8c\) et puisque l'algèbre \(\CC[\eisen_4, \eisen_6]\) des formes modulaires pour \(\sldz\) est engendrée par une fonction de poids \(4\) et une de poids \(6\), on en déduit que
\begin{equation}\label{eq_frommodular}
\dim\JS{k}=\sum_{c=0}^{\Ent{k/4}}\fd(2k-8c)
\end{equation}
où pour tout \(j\in\N\) on a noté \(\fd(j)\) la dimension de l'espace des formes modulaires de poids \(j\), explicitement donnée par
\begin{equation}\label{eq_d}
\fd(j)=\Ent*{\frac{j}{12}}+\begin{cases*}
0 & si \(12\) divise \(j-2\) \\
1 & sinon. 
\end{cases*}
\end{equation}
Même s'il n'existe pas de formes modulaires de poids négatifs, et que \(\fd(j)\) devrait être choisi nul pour \(j<0\), on prend une autre convention pour mener la suite des calculs qui ne font pas intervenir l'aspect modulaire de \(\fd\) mais seulement son aspect combinatoire. On étend la définition de \(\fd\) par~\eqref{eq_d} à \(\Z\) tout entier. On a alors \(\fd(j+12)=\fd(j)+1\) pour tout \(j\in\Z\). 

Soit \(x\) un réel, on note \(\projZ{x}\) l'entier le plus proche de \(x\) (avec la convention \(\projZ{n+1/2}=n\) pour tout \(n \in \Z\)). 
\begin{prop}\label{prop_dimJS}
Pour tout entier naturel \(k\), la dimension \(\fds(k)\) de l'espace des formes de Jacobi singulières d'indice nul et de poids \(k\) est donnée par 
\begin{equation}\label{eq_dimexplicit}
\fds(k)=\dim\JS{k}=
\projZ*{\frac{\left(k+3\olddelta(k)\right)^2}{48}}\qquad\text{avec}\qquad\olddelta(k)=\begin{dcases*} 1 & si \(k\) est impair\\ 2 & sinon \end{dcases*}
\end{equation}
La série génératrice de ces dimensions est
\[
\sum_{k\in\N}\fds(k)\cdot z^k=\frac{1}{(1-z^2)(1-z^3)(1-z^4)}
\]
et on a les relations de récurrence : 
\[
\fds(2k+3)=\fds(2k)\qquad\text{et}\qquad\fds(2k+13)=\fds(2k+1)+k+5
\]
pour tout entier \(k\). Les premières valeurs sont données par
\[
\begin{array}{|c||c|c|c|c|c|c|c|c|}
\hline
k & 0 & 1 & 2  & 4 & 6 & 8 & 10 & 12\\
\hline
\fds(k) & 1 & 0 & 1 & 2 & 3 & 4 & 5 & 7\\
\hline
\end{array}
\]

\end{prop}
\begin{proof}
En comptant les éléments de la base~\eqref{eq_base} de \(\JS{k}\), on trouve que la série génératrice de \(\fds\) est
\[
\sum_{k\in\N}\#\left\{(a,b,c)\in\N^3\colon k=2a+3b+4c\right\}z^k=\sum_{a\in\N}z^{2a}\sum_{b\in\N}z^{3b}\sum_{c\in\N}z^{4c}=\frac{1}{(1-z^2)(1-z^3)(1-z^4)}.
\]
On en déduit ensuite
\[
\sum_{k\in\N}\fds(2k)z^k=\frac{1}{(1-z)^3(1+z)(1+z+z^2)}
\]
puis
\[
\sum_{k\in\N}\fds(2k+1)z^k=\frac{z}{(1-z)^3(1+z)(1+z+z^2)}. 
\]
On en tire immédiatement que
\begin{equation}\label{eq_pairimpair}
\fds(2k)=\fds(2k+3)
\end{equation}
pour tout entier \(k\). Compte-tenu de~\eqref{eq_frommodular} et de l'extension à \(\Z\) de~\eqref{eq_d}, on a
\[
\fds(2k+13)=\fds(2k+1)+\sum_{c=\Ent*{\frac{2k+1}{4}}+1}^{\Ent*{\frac{2k+1}{4}}+3}\fd(4k-8c+2)+2\Ent*{\frac{2k+1}{4}}+8. 
\]
On en déduit la deuxième relation de récurrence :
\begin{equation}\label{eq_impairimpair}
\fds(2k+13)=\fds(2k+1)+k+5
\end{equation}
pour tout entier \(k\). 

L'application \(\varphi : k\mapsto\frac{\left(k+3\olddelta(k)\right)^2}{48}\) vérifie elle aussi les relations~\eqref{eq_pairimpair} et~\eqref{eq_impairimpair}, il est est donc de même de \(\projZ{\varphi}\).  On conclut que
\(\displaystyle\fds(k)=
\projZ*{\frac{\left(k+3\olddelta(k)\right)^2}{48}}\)
pour tout entier \(k\), en comparant les valeurs pour \(k\in\{0,1,2,4,6,8,10,12\}\),
\end{proof}
\begin{rem}
On déduit de la série génératrice de \(\left(\fds(k)\right)_{k\in\N}\) que cette suite est \(\left(t(k+3)\right)_{k\in\N}\) où \(t\) est la suite d'Alcuin~\cite{AlcuinSeq}. La formule explicite est alors démontrée dans~\cite[\texten{Theorem}~1]{zbMATH06071095}. Les équations~\eqref{eq_pairimpair} et~\eqref{eq_impairimpair}  sont données dans ce contexte dans~\cite{zbMATH03646988} et démontrées dans~\cite{zbMATH03654914}.
\end{rem}
\begin{rem}
On peut obtenir de façon systématique des formules similaires pour les dimensions des espaces considérés dans ce texte. Une discussion sur ces formules est proposée en annexe. 
\end{rem}
\subsubsection{Application à l'équation différentielle de la fonction de Weierstra\ss{}}
La forme modulaire \(\eisen_6\) est une forme de Jacobi singulière d'indice nul et de poids \(6\). La dimension de \(\JS{6}\) est \(3\), une base étant  \(\left((\dz\wp)^2,\wp^3,\wp\eisen_4\right)\). Ainsi \(\eisen_6\) est combinaison linéaire de ces trois fonctions. En identifiant les termes en \(z^{-6}\), \(z^{-2}\) et \(z^0\) du développement de Laurent en \(z=0\), on obtient :
\begin{equation}\label{eq_diffwp}
\eisen_6=-\frac{1}{140}(\dz\wp)^2+\frac{1}{35}\wp^3-\frac{3}{7}\wp\eisen_4. 
\end{equation}
On retrouve ainsi l'équation différentielle de la fonction \(\wp\) de Weierstra\ss{} au cœur de la théorie des courbes elliptiques\cite[\texten{Theorem}~V.3.4]{zbMATH05500775}. 

\section{Formes singulières quasi-Jacobi d'indice nul}\label{sec_trois}
\subsection{Action et dérivations}
L'action de \(\Jac\) sur \(\pkC\) est donnée par l'application \(\AC\) : 
\begin{equation}\label{eq_defAC}
\begin{array}{ccccc}
\AC	& \colon	& \Jac	& \to		&	\left(\pkC\right)^{\pkC}\\
	&		& A=(g, \Lambda)=\left(\begin{pmatrix}a & b\\ c & d\end{pmatrix},(\lambda,\mu)\right)		& \mapsto	&	\begin{array}{ccc}
									\pkC		& \to		& \pkC\\
									(\tau,z)	& \mapsto	& A\cdot(\tau,z)=\left(\frac{a\tau+b}{c\tau+d},\frac{z+\lambda\tau+\mu}{c\tau+d}\right). 
								\end{array}	 
\end{array}								
\end{equation}
Par définition d'une action, on a
\begin{equation}\label{eq_Acac}
\AC(AB)=\AC(A)\circ\AC(B).
\end{equation}
On calcule
\begin{equation}\label{eq_derivH}
\frac{\partial\AC}{\partial\tau}=\left(\frac{1}{\J^2},-\frac{\Y}{\J}\right)\quad\text{et}\quad\frac{\partial\AC}{\partial z}=\left(0,\frac{1}{\J}\right)
\end{equation}
avec
\[
\begin{array}{ccccc}
\J	& \colon	& \Jac	& \to		&	\fonc\\
	&		& \left(\begin{psmallmatrix}a & b\\c & d\end{psmallmatrix},(\lambda,\mu)\right)		& \mapsto	&	\begin{array}{ccc}
									\pkC		& \to		& \CC\\
									(\tau,z)	& \mapsto	& c\tau+d
								\end{array}	 
\end{array}								
\]
et
\[
\begin{array}{ccccc}
\Y	& \colon	& \Jac	& \to		&	\fonc\\
	&		& \left(\begin{psmallmatrix}a & b\\c & d\end{psmallmatrix},(\lambda,\mu)\right)		& \mapsto	&	\begin{array}{ccc}
									\pkC		& \to		& \CC\\
									(\tau,z)	& \mapsto	& \dfrac{cz+c\mu -d\lambda}{c\tau+d}. 
								\end{array}	 
\end{array}								
\]
En définissant
\[
\begin{array}{ccccc}
\X	& \colon	& \Jac	& \to		&	\fonc\\
	&		& \left(\begin{psmallmatrix}a & b\\c & d\end{psmallmatrix},(\lambda,\mu)\right)		& \mapsto	&	\begin{array}{ccc}
									\pkC		& \to		& \CC\\
									(\tau,z)	& \mapsto	& \dfrac{c}{c\tau+d}
								\end{array}	 
\end{array}								
\]
on a 
\begin{subequations}\label{eq_derivees}
\begin{alignat}{5}
\frac{\partial\J}{\partial\tau}	& =\X\J	&\qquad\qquad& &\frac{\partial\Y}{\partial\tau}	& =-\X\Y	&\qquad\qquad& & \frac{\partial\X}{\partial\tau}	&=-\X^2\label{eq_deriveestau}\\
\frac{\partial\J}{\partial z}		&=0			&& &\frac{\partial\Y}{\partial z}	&=\X	&& &\frac{\partial\X}{\partial z}	&=0.\label{eq_deriveesz}
\end{alignat}
\end{subequations}
Il est clair que les applications \(\J\), \(X\) et \(\Y\) sont algébriquement indépendantes sur \(\CC\).

Il résulte de \eqref{eq_derivees} que l'algèbre \(\CC[\J,\X,\Y]\) est donc stable par les dérivations en \(\tau\) et \(z\). 
La preuve de la proposition suivante montre que la notion de cocycle permet d'appréhender les dérivations de l'action. 
\begin{prop}\label{prop_JXYcocycles}

On a, pour les applications \(\J\), \(\X\) et \(\Y\) et l'action définie en \eqref{eq_vertkm}, les relations de 1-cocyclicité suivantes : \(\forall(A, B)\in\left(\sldz\ltimes\Z^2\right)^2 \)
\[ 
\J(AB)=\left(\J(A)\vert_{0,0}B\right)\J(B), \quad \Y(AB)=\Y(A)\vert_{1,0}B+\Y(B),\quad \X(AB)=\X(A)\vert_{2,0}B+\X(B).
\]
\end{prop}
\begin{proof}
La première relation sur \(\J\) est bien connue et facile à vérifier. 
Pour la deuxième formule, dérivons~\eqref{eq_Acac} par rapport à \(\tau\). En notant \(\AC=(\AC_1,\AC_2)\), on trouve :
\[
\frac{1}{\J(AB)}\left(\frac{1}{\J(AB)},-\Y(AB)\right)=\frac{\partial\AC(A)}{\partial\tau}\left(\AC(B)\right)\frac{\partial\AC_1(B)}{\partial\tau}+\frac{\partial\AC(A)}{\partial z}\left(\AC(B)\right)\frac{\partial\AC_2(B)}{\partial\tau}
\]
ce qui, en utilisant~\eqref{eq_derivH}, conduit à
\[
\left(\frac{1}{\J(AB)(x)^2},-\frac{\Y(AB)(x)}{\J(AB)(x)}\right)=\left(\frac{1}{\J(A)(Bx)^2\J(B)(x)^2},-\frac{\Y(A)(Bx)}{\J(A)(Bx)\J(B)(x)^2}-\frac{\Y(B)(x)}{\J(A)(Bx)\J(B)(x)}\right)
\]
où l'on a noté \(x=(\tau,z)\). En comparant les deuxièmes coordonnées et en utilisant la formule précédente, on obtient 
\( \Y(AB)(x)=\J(B)(x)^{-1}\Y(A)(Bx)+\Y(B)(x)\)
ce qui montre la relation voulue.

Dérivons ensuite par rapport à \(z\) la relation de cocyclicité de \(\Y\), on trouve
\[
\frac{\partial\Y(AB)}{\partial z}(x)=\frac{1}{\J(B)(x)}\frac{\partial\Y(A)}{\partial\tau}(Bx)\frac{\partial\AC_1(B)}{\partial z}(x)+\frac{1}{\J(B)(x)}\frac{\partial\Y(A)}{\partial z}(Bx)\frac{\partial\AC_2(B)}{\partial z}(x)+\frac{\partial\Y(B)}{\partial z}(x).
\]
Grâce à~\eqref{eq_deriveesz} et~\eqref{eq_derivH}, on en déduit
\[
\X(AB)(x)=\J(B)(x)^{-2}\X(A)(Bx)+\X(B)(x). 
\]
C'est la relation de cocyclicité de \(\X\). 
\end{proof}

\subsection{Définition}
\begin{dfn}\label{dfn_JS}
Une fonction singulière \(f\colon\pkC\to\CC\) est une forme \emph{singulière quasi-Jacobi} (d'indice nul), de poids \(k\in\N\) et de profondeur \((s_1,s_2)\in\N^2\) s'il existe \(\left(f_{j_1,j_2}\right)_{\substack{0\leq j_1\leq s_1\\ 0\leq j_2\leq s_2}}\in\Sing^{(s_1+1)(s_2+1)}\) telle que
\begin{equation}\label{eq_trsfJac}
\forall A\in\Jac\qquad f\vert_{k,0}A=\sum_{j_1=0}^{s_1}\sum_{j_2=0}^{s_2}f_{j_1,j_2}\X(A)^{j_1}\Y(A)^{j_2}. 
\end{equation}
où \(f_{s_1, s_2}\) est non identiquement nulle. On convient de noter désormais \(f\vert_kA :=f\vert_{k, 0}A\), on ne considérera que des formes d'indice nul.
Il résulte de l'indépendance algébrique de \(\X\) et \(\Y\) sur \(\CC\) que la décomposition \eqref{eq_trsfJac} est unique. On note alors \(\cQ_{j_1,j_2}(f)=f_{j_1,j_2}\), et on appelle \(s_1\) la \emph{profondeur modulaire} de \(f\) et \(s_2\) sa \emph{profondeur elliptique}.  L'espace vectoriel des formes singulières quasi-Jacobi de poids \(k\) et profondeurs inférieures ou égales à \(s_1\) et \(s_2\) est noté \(\QJS{k}{s_1}{s_2}\) ; l'espace vectoriel des formes singulières quasi-Jacobi de poids \(k\) est noté \(\QJSpoids{k}\). 
\end{dfn}

\AjoutRouge{
\begin{rem}
La notion introduite ici est cohérente avec celle de~\cite[\S2.4]{zbMATH07634708}. Alors que cette article expose une approche \emph{via} la notion de \emph{forme presque Jacobi}, nous privilégions une approche directe. 
\end{rem}
}

\begin{rem}\label{jacobisinguliereprofnulle}\leavevmode
Le choix \(A=\left(\begin{psmallmatrix}1 & 0\\0 & 1\end{psmallmatrix},(0,0)\right)\) implique que \(\cQ_{0,0}(f)=f\). Cela entra\^\i ne en particulier que \(\QJS{k}{0}{0}\) est l'espace \(\JS{k}\) des formes de Jacobi singulières d'indice nul et de poids \(k\) rencontrées précédemment. 

\end{rem}
\begin{rem}
\begin{itemize}
 \item 
Soit \(f \in \QJSpoids{k}\) et \(g \in \QJSpoids{\ell}\), alors on a \(fg \in \QJSpoids{k+\ell}\) et 
\[ 
\cQ_{i,j}(fg)=\sum_{\substack{(\alpha,\beta,\gamma,\delta)\\\alpha+\beta=i\\\gamma+\delta=j}}\cQ_{\alpha,\gamma}(f)\cQ_{\beta,\delta}(g). 
\]
\item Il résulte de l'indépendance algébrique de \(\X, \Y\) et \(\J\) sur \(\CC\) que les espaces \(\QJSpoids{k}\) sont en somme directe. On peut donc considérer l'algèbre graduée par le poids \(\displaystyle\AlgQJS=\bigoplus_{k \in \N}\QJSpoids{k}\), qu'on conviendra d'appeler algèbre des formes singulières quasi-Jacobi.
\end{itemize}
\end{rem}

\subsection{Stabilité par dérivation}
La dérivation par rapport à \(z\) est nulle sur l'algèbre \(\AlgM\) des formes modulaires. En revanche, \(\AlgM\) n'est pas stable par dérivation par rapport à \(\tau\), ce qui justifie l'introduction de l'algèbre \(\AlgQM\) des formes quasimodulaires\cite{zbMATH06128504,zbMATH05050117}. 

L'algèbre \(\AlgJS\) des formes de Jacobi singulières est stable par dérivation par rapport à \(z\) mais n'est pas stable par dérivation par rapport à \(\tau\) (comme on le verra plus tard, voir remarque~\vref{rem_profdz}, \eqref{eq_ddeuxzrho} et  \eqref{eq_dtaudzrho}). Nous montrons ici que l'algèbre \(\AlgQJS\) est stable par chacune de ces dérivations. 
\begin{lem}\label{lem_derivAC}
Soit \(f\colon\pkC\to\CC\) dérivable par rapport à chaque variable, alors
\begin{equation}\label{eq_derivACz}
\frac{\partial\left(f\vert_kA\right)}{\partial z}=\left.\left(\frac{\partial f}{\partial z}\right)\right\rvert_{k+1} A
\end{equation}
et
\begin{equation}\label{eq_derivACtau}
\frac{\partial\left(f\vert_kA\right)}{\partial\tau}=-k\left(f\vert_kA\right)\X(A)+\left.\left(\frac{\partial f}{\partial\tau}\right)\right\rvert_{k+2} A-\Y(A)\left.\left(\frac{\partial f}{\partial z}\right)\right\rvert_{k+1} A. 
\end{equation}
\end{lem}
\begin{proof}
Le résultat est obtenu en dérivant par rapport à \(z\) et à \(\tau\) la définition \(f\vert_kA=\J(A)^{-k}f\left(\AC(A)\right)\) puis en utilisant~\eqref{eq_derivH} et~\eqref{eq_derivees}. 
\end{proof}
\begin{prop}\label{prop_stabder}
L'algèbre \(\AlgQJS\) est stable par les dérivations en \(z\) et \(\tau\). La dérivation \(\partial/\partial z\) envoie \(\QJS{k}{s_1}{s_2}\) dans \(\QJS{k+1}{s_1+1}{s_2}\) ; la dérivation \(\partial/\partial\tau\) envoie \(\QJS{k}{s_1}{s_2}\) dans \(\QJS{k+2}{s_1+1}{s_2+1}\). D'autre part, pour \(f \in \QJSpoids{k}\),
\[
\cQ_{j_1,j_2}\left(\frac{\partial f}{\partial z}\right)=\frac{\partial\cQ_{j_1,j_2}(f)}{\partial z}+(j_2+1)\cQ_{j_1-1,j_2+1}(f)
\]
et
\[
\cQ_{j_1,j_2}\left(\frac{\partial f}{\partial\tau}\right)=\frac{\partial\cQ_{j_1,j_2}(f)}{\partial\tau}+\frac{\partial\cQ_{j_1,j_2-1}(f)}{\partial z}+(k-j_1+1)\cQ_{j_1-1,j_2}(f).
\]
Plus précisément,
\[
\frac{\partial}{\partial z}\QJS{k}{s_1}{s_2}\subseteq\QJS{k+1}{s_1+1}{s_2-1}+\QJS{k+1}{s_1}{s_2}
\]
et
\[
\frac{\partial}{\partial\tau}\QJS{k}{s_1}{s_2}\subseteq\QJS{k+2}{s_1+1}{s_2}+\QJS{k+2}{s_1}{s_2+1}.
\]
\end{prop}
\begin{proof}
Grâce à~\eqref{eq_derivACz} et à la définition~\ref{dfn_JS}, on trouve
\[
\left.\left(\frac{\partial f}{\partial z}\right)\right\rvert_{k+1} A=\sum_{j_1=0}^{s_1}\sum_{j_2=0}^{s_2}\left(\frac{\partial f_{j_1,j_2}}{\partial z}\X(A)^{j_1}\Y(A)^{j_2}+j_2f_{j_1,j_2}\X(A)^{j_1+1}\Y(A)^{j_2-1}\right). 
\]
On en déduit les résultats relatifs à \(\partial/\partial z\). 

D'autre part, grâce à~\eqref{eq_derivACtau} et à la définition~\ref{dfn_JS}, on trouve
\begin{multline*}
-k\left(f\vert_kA\right)\X(A)+\left.\left(\frac{\partial f}{\partial\tau}\right)\right\rvert_{k+2} A-\Y(A)\left.\left(\frac{\partial f}{\partial z}\right)\right\rvert_{k+1} A=\\
\sum_{j_1=0}^{s_1}\sum_{j_2=0}^{s_2}\left(
\frac{\partial f_{j_1,j_2}}{\partial\tau}\X(A)^{j_1}\Y(A)^{j_2}-j_1f_{j_1,j_2}\X(A)^{j_1+1}\Y(A)^{j_2}-j_2f_{j_1,j_2}\X(A)^{j_1+1}\Y(A)^{j_2}
\right).  
\end{multline*}
En utilisant les résultats relatifs à \(\partial/\partial z\), on trouve alors ceux relatifs à \(\partial/\partial\tau\). 

Si \(f\in\QJS{k}{s_1}{s_2}\), alors \(\frac{\partial f}{\partial z}\in\QJS{k+1}{s_1+1}{s_2}\) mais \(\cQ_{s_1+1,s_2}(\partial f/\partial z)=0\) donc \(\frac{\partial f}{\partial z}\in\QJS{k+1}{s_1+1}{s_2-1}+\QJS{k+1}{s_1}{s_2}\). L'inclusion pour \(\partial/\partial\tau\) se démontre de la même façon. 
\end{proof}
\begin{rem}\label{rem_profdz}
Ainsi, si \(s_2=0\) alors \(\frac{\partial f}{\partial z}\in\QJS{k+1}{s_1}{0}\). En particulier, si \(f \in \AlgJS_k, \dfrac{\partial f}{\partial z} \in \AlgJS_{k+1}\).
\end{rem}
\subsection{Exemples fondamentaux}\label{subsec_qjacfonda}
Les résultats de cette partie sont résumés dans la table~\vref{tab_explefonda}.
\subsubsection{Formes quasimodulaires}\label{subsub_fondafqmod} Comme dit dans le paragraphe \ref{subsec_jacosingexple}, on identifie dans la suite tout fonction \(f : \pk \to \CC\) avec la fonction \(f\colon\pkC\to\CC\) définie par \(f(\tau,z)=f(\tau)\). Via cette identification, toute forme modulaire de poids \(k\) est une forme singulière quasi-Jacobi de poids \(k\) et profondeur \((0,0)\). La dérivée \(n\)-ième (en \(\tau)\) d'une forme modulaire de poids \(k\) est alors une forme singulière quasi-Jacobi de poids \(k+2n\) et profondeur \((n,0)\). De même, \(\eisen_2\) est une forme singulière quasi-Jacobi de poids \(2\) et profondeur \((1,0)\) avec \(\cQ_{1,0}\left(\eisen_2\right)=-2\ic\pi\). L'algèbre des formes quasimodulaires étant engendrée par les formes modulaires \(\eisen_4\) et \(\eisen_6\) et par la forme quasimodulaire \(\eisen_2\), on a donc montré que les formes quasimodulaires sont toutes des formes singulières quasi-Jacobi. 

\subsubsection{La première fonction d'Eisenstein décalée} La série d'Eisenstein décalée de poids \(1\) est la série définie sur \(\pkC\) par
\[
\Eisen_1(\tau,z)=\lim_{M\to+\infty}\sum_{m=-M}^M\left(\lim_{N\to+\infty}\sum_{\substack{n=-N\\ m=0\Rightarrow n\neq 0}}^N\frac{1}{z+m+n\tau}\right)
\]
\cite[\texten{Chapter}~III, \S 2]{zbMATH01236956}. Cette fonction est bien définie et admet un développement en série de Laurent 
\begin{equation}\label{eq_LaurentEun}
\Eisen_1(\tau,z)=\frac{1}{z}-\sum_{n=0}^{+\infty}\eisen_{2n+2}(\tau)z^{2n+1}
\end{equation}
la série convergeant sur tout disque ouvert épointé de centre \(z=0\) de rayon inférieur à \(\abs{\tau}\) \cite[\texten{Chapter}~III, eq. (9)]{zbMATH01236956}. Elle vérifie l'équation :
\[
\forall A\in\Jac\qquad\Eisen_1\vert_1A=\Eisen_1+2\ic\pi\Y(A)
\]
\cite[\texten{Lemma}~1]{hal03132764}\footnote{Dans ce travail, on a noté \(\mathrm{J}_1\) ce que l'on note ici \(\frac{1}{2i\pi}\Eisen_1\).} ; la fonction \(z\mapsto\Eisen_1(\tau,z)\) est méromorphe, ses pôles sont les points du réseau \(\Z+\tau\Z\) et ils sont simples. La fonction \(\Eisen_1\) est donc une forme singulière quasi-Jacobi de poids \(1\) et profondeur \((0,1)\).

\AjoutRouge{
\begin{rem}
\Ajout{Nous utilisons la convention de Weil dans~\cite{zbMATH07634708} qui réserve les lettres majuscules aux fonctions intrinsèquement de deux variables et les lettres minuscules aux fonctions intrinsèquement d'une variable.} La fonction \(\Eisen_1\) et les généralisations obtenues en remplaçant \(z+m+n\tau\) par \((z+m+n\tau)^k\) sont des cas particuliers de ce que Charollois \& Sczech~\cite{zbMATH06696479} appelent séries de Kronecker-Eisenstein. Notre fonction \(\Eisen_1\) est leur fonction \(K^*(z,0,1,\tau)=Ser(0,0,z,\tau)\). 
\end{rem}
}

\begin{table}
\begin{center}
\begin{tabular}{|>{$}c<{$}|>{$}c<{$}|>{$}c<{$}|>{$}c<{$}|}
\hline
\text{Fonction} & \text{Poids} & \begin{tabular}{c}Profondeur \\ \((s_1,s_2)\)\end{tabular} & \cQ_{s_1,s_2}\\
\hline
\wp & 2 & (0,0) & \wp\\
\hline
\dz\wp & 3 & (0,0) & \dz\wp\\
\hline
\eisen_4 & 4 & (0,0) & \eisen_4\\
\hline
\Eisen_1 & 1 & (0,1) & 2\ic\pi\\
\hline
\eisen_2 & 2 & (1,0) & -2\ic\pi\\
\hline
\end{tabular}
\caption{Exemples fondamentaux de formes singulières quasi-Jacobi.}
\label{tab_explefonda}
\end{center}
\end{table}

\begin{lem}\label{lem_indepalgcinq}
Les fonctions \(\wp\), \(\dz\wp\), \(\eisen_4\), \(\Eisen_1\) et \(\eisen_2\) sont algébriquement indépendantes. 
\end{lem}
\begin{proof}
Grâce au théorème~\ref{prop_strucJS}, il suffit de montrer que si \(k\), \(s_1\) et \(s_2\) sont des entiers et si les \(f_{j_1,j_2}\) sont des formes de Jacobi singulières de poids \(k-j_1-2j_2\) telles que
\begin{equation}\label{eq_uniciprof}
\sum_{j_1=0}^{s_1}\sum_{j_2=0}^{s_2}f_{j_1,j_2}\Eisen_1^{j_1}\eisen_2^{j_2}=0
\end{equation}
alors, les \(f_{j_1,j_2}\) sont toutes nulles. Supposons par l'absurde que l'une est non nulle, on peut supposer que c'est \(f_{s_1,s_2}\). Alors, le terme de gauche de~\eqref{eq_uniciprof} est de profondeur \((s_1,s_2)\). Par unicité de la profondeur, on déduit que \(s_1=s_2=0\) puisque le membre de droite est de profondeur nulle puis la nullité de tous les \(f_{j_1,j_2}\). 
\end{proof}

\subsection{Structure}
La partie~\ref{subsec_qjacfonda} montre \(\CC[\wp,\dz\wp,\eisen_4,\Eisen_1,\eisen_2]\subseteq\AlgQJS\). L'objectif de cette partie est de montrer l'égalité des deux algèbres. 

La preuve repose sur le lemme suivant. 
\begin{lem}\label{lem_coefestforme}
Soit \(f\) une forme singulière quasi-Jacobi de poids \(k\) et profondeur \((s_1,s_2)\). Alors \(\cQ_{s_1,s_2}(f)\) est une forme de Jacobi singulière de poids \(k-2s_1-s_2\). 
\end{lem}
\begin{proof}
Si \(A\) et \(B\) sont deux éléments de \(\Jac\), on a d'une part
\begin{equation}\label{eq_dbleacun}
f\vert_k(AB)=\sum_{x=0}^{s_1}\sum_{y=0}^{s_2}\cQ_{x,y}(f)\X(AB)^{x}\Y(AB)^{y}
\end{equation}
et d'autre part
\[
f\vert_k(AB)=\left(f\vert_kA\right)\vert_kB=\sum_{j_1=0}^{s_1}\sum_{j_2=0}^{s_2}\left(\cQ_{j_1,j_2}(f)\vert_{k-2j_1-j_2}B\right)\left(\X(A)\vert_2B\right)^{j_1}\left(\Y(A)\vert_1B\right)^{j_2}.
\]
Pour transformer cette dernière égalité, on utilise la proposition~\ref{prop_JXYcocycles} pour obtenir
\begin{multline}\label{eq_dbleacdeux}
f\vert_k(AB)=\sum_{x=0}^{s_1}\sum_{y=0}^{s_2}\left(\sum_{j_1=x}^{s_1}\sum_{j_2=y}^{s_2}\binom{j_1}{x}\binom{j_2}{y}\left(-\X(B)\right)^{j_1-x}\left(-\Y(B)\right)^{j_2-y}\right)\left(\cQ_{j_1,j_2}(f)\vert_{k-2j_1-j_2}B\right)\\\X(AB)^x\Y(AB)^y. 
\end{multline}
Comparant les coefficients de \(\X(AB)^{s_1}\Y(AB)^{s_2}\) dans~\eqref{eq_dbleacun} et~\eqref{eq_dbleacdeux}, on trouve
\[
\cQ_{s_1,s_2}(f)\vert_{k-2s_1-s_2}B=\cQ_{s_1,s_2}(f). 
\]
Puisque \(\cQ_{s_1,s_2}(f)\) est singulière, on en déduit que \(\cQ_{s_1,s_2}(f)\) est une forme de Jacobi singulière de poids \(k-2s_1-s_2\). 
\end{proof}

\begin{thm}\label{thm_strucQJ}
L'algèbre des formes singulières quasi-Jacobi est engendrée par les fonctions \(\wp\),\(\dz\wp\),\(\eisen_4\),\(\Eisen_1\) et \(\eisen_2\). On a donc
\[
\AlgQJS=\CC[\wp,\dz\wp,\eisen_4,\Eisen_1,\eisen_2]. 
\]
\end{thm}
\begin{proof}
On a montré (voir le théorème~\ref{prop_strucJS}) que \(\AlgJS=\CC[\wp,\dz\wp,\eisen_4]\). Soit \(f\in\QJS{k}{s_1}{s_2}\), posons
\[
g=f-(-1)^{s_1}\left(\frac{1}{2\ic\pi}\right)^{s_1+s_2}\cQ_{s_1,s_2}(f)\eisen_2^{s_1}\Eisen_1^{s_2}. 
\]
Alors
\begin{enumerate}
\item  \(g\in\QJS{k}{s_1-1}{s_2}+\QJS{k}{s_1}{s_2-1}\) ;
\item \(\cQ_{s_1,s_2}(f)\in\JS{k-2s_1-s_2}\subset\CC[\wp,\dz\wp,\eisen_4]\) d'après le lemme~\ref{lem_coefestforme} et donc \(g-f\in\CC[\wp,\dz\wp,\eisen_4,\Eisen_1,\eisen_2]\). 
\end{enumerate}
A partir de la remarque~\ref{jacobisinguliereprofnulle}, par récurrence sur \(s_1+s_2\), on obtient 
\[
\forall k\in\N\enspace\forall(s_1,s_2)\in\N^2\qquad \QJS{k}{s_1}{s_2}\subseteq\CC[\wp,\dz\wp,\eisen_4,\Eisen_1,\eisen_2]. 
\]
D'après le lemme \ref{lem_indepalgcinq}, \(\AlgQJS\) est donc l'algèbre de polyn\^omes \(\CC[\wp,\dz\wp,\eisen_4,\Eisen_1,\eisen_2]\).
\end{proof}
\subsection{Sous-algèbres remarquables}\label{sec_sousalgrem}
Les résultats de cette partie sont résumés figure~\vref{fig_diagram}
\subsubsection{Formes quasi-Jacobi de type quasielliptique}\label{subsec_jacell}
\begin{dfn}
On appelle \emph{forme quasi-Jacobi de type quasielliptique} de poids~\(k\) et profondeur~\(s\) toute forme singulière quasi-Jacobi de poids \(k\) et profondeur \((0,s)\). 
\end{dfn}
On note \(\mathrm{JS}_k^{0,\leq s}\) l'espace vectoriel de telles formes de profondeur inférieure ou égale à \(s\). On note alors \(\displaystyle\AlgQJell = \bigoplus_{k=0}^{\infty} \bigcup_{s \geq 0} \mathrm{JS}_k^{0,\leq s}\) qu'on convient d'appeler dans la suite ensemble des formes quasi-Jacobi de type quasielliptique. 

Grâce au théorème~\ref{thm_strucQJ}, c'est une algèbre de polyn\^omes~:
\[
\AlgQJell=\CC[\wp,\dz\wp,\eisen_4,\Eisen_1].
\]
On a \(\AlgM \subset \AlgJS \subset \AlgQJell \subset \AlgQJS\). 

L'équation~\eqref{eq_dtaurho} montre que \(\AlgQJell\) n'est pas stable par la \emph{dérivation modulaire}
\[
\dtau=\frac{\pi}{2\ic}\frac{\partial}{\partial\tau}. 
\]
D'après les équations~\eqref{eq_delLaurent} et~\eqref{eq_LaurentEun}, on a 
\begin{equation}\label{eq_Euwped}
\frac{\partial\Eisen_1}{\partial z}=-\wp-\eisen_2
\end{equation}
et donc \(\AlgQJell\) n'est pas stable par la \emph{dérivation elliptique}
\[
\dz=\frac{\partial}{\partial z}.
\]

La table~\vref{tab_stabderalginter} récapitule les stabilités par dérivation des différentes algèbres en jeu. 
\subsubsection{Formes quasi-Jacobi de type quasimodulaire}\label{sec_fqjm}
\begin{dfn}
On appelle \emph{forme quasi-Jacobi de type quasimodulaire} de poids \(k\) et profondeur \(s\) toute forme singulière quasi-Jacobi de poids \(k\) et profondeur \((s,0)\). 
\end{dfn}
On note \(\mathrm{JS}_k^{\leq s, 0}\) l'espace vectoriel de telles formes de profondeur inférieure ou égale à \(s\). On note \(\displaystyle\AlgQJmod = \bigoplus_{k=0}^{\infty} \bigcup_{s \geq 0} \mathrm{JS}_k^{\leq s, 0}\) qu'on convient d'appeler dans la suite ensemble des formes quasi-Jacobi de type quasimodulaire. 

Grâce au théorème~\ref{thm_strucQJ}, c'est une algèbre de polynômes~:
\[
\AlgQJmod=\CC[\wp,\dz\wp,\eisen_4,\eisen_2].
\]
On a \(\AlgM \subset \AlgJS \subset \AlgQJmod \subset \AlgQJS\) et \(\AlgM \subset \AlgQM \subset \AlgQJmod \subset \AlgQJS\). 

 Grâce à la remarque~\ref{rem_profdz}, l'algèbre \(\AlgQJmod\) est stable par la dérivation \(\dz\). L'équation~\eqref{eq_dtaurho} montre qu'elle n'est pas stable par la dérivation \(\dtau\).  

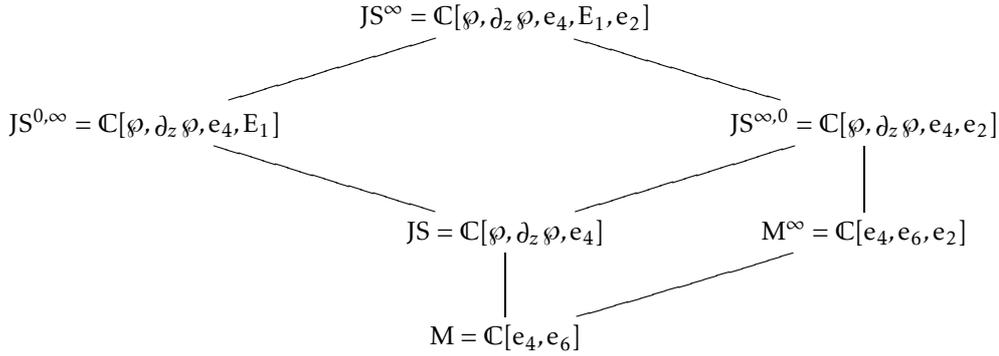
\begin{figure}
\begin{center}
\begin{equation*}
\xymatrix{ & \AlgQJS=\CC[\wp,\dz\wp,\eisen_4,\Eisen_1,\eisen_2] \ar@{-}[rd] \\ 
\AlgQJell=\CC[\wp,\dz\wp,\eisen_4,\Eisen_1] \ar@{-}[ru] \ar@{-}[rd] &&\AlgQJmod=\CC[\wp,\dz\wp,\eisen_4,\eisen_2]   \\ 
&\AlgJS=\CC[\wp,\dz\wp,\eisen_4] \ar@{-}[ru] & \AlgQM=\CC[\eisen_4,\eisen_6,\eisen_2]\ar@{-}[u] \\
& \ar@{-}[u] \AlgM=\CC[\eisen_4,\eisen_6] \ar@{-}[ru]&  }
\end{equation*}
\caption{Sous-algèbres remarquables.}
\label{fig_diagram}
\end{center}
\end{figure}

\subsection{Équations différentielles fondamentales} 
\subsubsection{Dérivée d'Oberdieck}
\begin{dfn}
On définit une dérivation sur l'algèbre \(\AlgQJS\) des formes singulières quasi-Jacobi en prolongeant par linéarité la définition suivante  
\[
\text{pour tout forme \(f\in\QJSpoids{k}\)},\qquad\Ob(f)=4\dtau(f)+\Eisen_1\dz(f)-k\eisen_2f. 
\]
On appelle \emph{dérivation d'Oberdieck} cette dérivation. 
\end{dfn}
\begin{rem}
On a \(\Ob=4\pi^2\ObCDMR\) où \(\ObCDMR\) est définie dans~\cite{hal03132764}. Le nom fait référence au travail de Georg Oberdieck~\cite{oberdieck2014serrederivativeweightjacobi}. La restriction de \(\Ob\) à \(\AlgM\) est la dérivation de Serre.
\end{rem}
La dérivation \(\Ob\) envoie par définition \(\QJS{k}{s_1}{s_2}\) dans \(\QJS{k+2}{s_1+1}{s_2+1}\). On a plus précisément la proposition suivante~: 
\begin{prop}\label{prop_ObstabJS}

\begin{enumerate}
 \item La dérivation \(\Ob\) envoie \(\QJS{k}{s_1}{s_2}\) dans \(\QJS{k+2}{s_1+1}{s_2}\).
 \item L'algèbre \(\AlgJS\) est stable par \(\Ob\) : l'image par \(\Ob\) de toute forme de Jacobi singulière de poids \(k\) est une forme de Jacobi singulière de poids \(k+2\). 
\end{enumerate}
\end{prop}

\begin{proof}
Soit \(f\in\QJSpoids{k}\). En utilisant la proposition~\ref{prop_stabder}, on voit que
\begin{multline}\label{eq_Qob}
\cQ_{j_1,j_2}(\Ob(f))=4\dtau(\cQ_{j_1,j_2}(f))+\Eisen_1\dz(\cQ_{j_1,j_2}(f))-k\eisen_2\cQ_{j_1,j_2}(f)+2\ic\pi(j_1+j_2-1)\cQ_{j_1-1,j_2}(f)\\+(j_2+1)\Eisen_1\cQ_{j_1-1,j_2+1}(f). 
\end{multline}

Si \(f \in \QJS{k}{s_1}{s_2}\), on a donc \(\cQ_{j_1, s_2+1}(f)=0\) pour tout \(j_1\), donc \(\Ob(f) \in \QJS{k+2}{s_1+1}{s_2}\).

Si \(f\in\JS{k}\), on a \(\cQ_{1,0}\left(\Ob(f)\right)=0\) ce qui démontre que \(\Ob(f) \in \JS{k+2}\). 
\end{proof}
\begin{rem}
D'après la proposition~\ref{prop_ObstabJS}, la dérivation d'Oberdieck stabilise \(\AlgQJmod\). En revanche, comme on le verra à l'équation~\eqref{eq_ObEun}, elle ne stabilise pas \(\AlgQJell\). 
\end{rem}
\begin{table}\label{table_stabder}
\begin{center}
{\renewcommand{\arraystretch}{1.5}%
\begin{tabular}{|>{$}c<{$}|c|c|c|}
\hline
 & \(\dz\) & \(\dtau\) & \(\Ob\)\\
\hline
\AlgM & oui & non & oui\\
\hline
\AlgJS & oui & non & oui\\
\hline
\AlgQM & oui & oui & oui\\
\hline
\AlgQJell & non & non & non\\
\hline
\AlgQJmod & oui & non & oui\\
\hline
\AlgQJS & oui & oui & oui \\
\hline
\end{tabular}
}
\caption{Stabilité des algèbres par dérivation.}
\label{tab_stabderalginter}
\end{center}
\end{table}

\subsubsection{Applications}\label{subsubsec_application}
Les résultats généraux des paragraphes précédents permettent, en calculant explicitement les images par dérivation des générateurs \(\wp,\dz\wp,\eisen_4,\Eisen_1,\eisen_2\), de déterminer des relations différentielles entre ces générateurs. 

La fonction \(\wp\) est une forme de Jacobi singulière de poids \(2\) et \(\Ob(\wp)\) est donc une forme de Jacobi singulière de poids \(4\) ;  grâce à la proposition~\ref{prop_dimJS} la dimension de \(\JS{4}\) est \(2\), une base étant \((\wp^2,\eisen_4)\). En égalisant les coefficients en \(1/z^4\) et constant, on trouve
\[
\Ob(\wp)=-2(\wp^2-10\eisen_4). 
\]
d'où l'on déduit
\begin{equation}\label{eq_dtaurho}
-4\dtau\wp=\Eisen_1\dz\wp+2\wp^2-2\eisen_2\wp-20\eisen_4. 
\end{equation}
L'égalisation des coefficients en \(z^{2n}\) pour tout \(n\geq 1\) conduit ensuite à
\begin{multline}\label{eq_gunther}
2(2n+1)\dtau\eisen_{2n+2}\\=(n+1)(2n+1)\eisen_{2n+2}\eisen_2-(n+2)(2n+5)\eisen_{2n+4}+\sum_{\mathclap{\substack{a \geq 1, b \geq 1\\a+b=n}}}(2a+1)(a-2b-1)\eisen_{2a+2}\eisen_{2b+2}. 
\end{multline}
En particulier pour \(n=1\) et \(n=2\) (et compte-tenu de l'égalité \(\eisen_8=\frac{3}{7}\eisen_4^2\) qui est conséquence du fait que l'espace des formes modulaires de poids \(8\) est de dimension \(1\)) on retrouve à l'aide de \eqref{eq_diffwp} les équations de Ramanujan
\begin{subequations}
 \label{eq_ramanujan}
\begin{align}
\dtau\eisen_4 &=\eisen_4\eisen_2-\frac{7}{2}\eisen_6\label{eq_ramae4}\\
&=-\frac{1}{10}\wp^3+\frac{1}{40}\left(\dz\wp\right)^2+\frac{3}{2}\wp\eisen_4+\eisen_4\eisen_2\label{eq_ramae4sanse6}\\
\dtau\eisen_6&=\frac{3}{2}\eisen_6\eisen_2-\frac{15}{7}\eisen_4^2.\label{eq_ramae6} 
\end{align}
\end{subequations}
En particulier
\begin{equation}\label{eq_Obeq}
	\Ob(\eisen_4)=-14\eisen_6=-\frac{2}{5}\wp^3+6\wp\eisen_4+\frac{1}{10}\left(\dz\wp\right)^2.  
\end{equation}	

Grâce à la remarque~\ref{rem_profdz}, la fonction \(\dz^2\wp\) est une forme de Jacobi singulière de poids \(4\) et donc une combinaison linéaire de \(\wp^2\) et \(\eisen_4\). En égalisant les termes en \(z^{-4}\) et constant du développement de Laurent, on obtient :
\begin{equation}\label{eq_ddeuxzrho}
\dz^2\wp=6(\wp^2-5\eisen_4). 
\end{equation}

La fonction \(\dz\wp\) est une forme de Jacobi singulière de poids \(3\) et \(\Ob(\dz\wp)\) est donc une forme de Jacobi singulière de poids \(5\) ;  l'espace \(\JS{5}\) est de dimension \(1\) engendré par \(\wp\dz\wp\). En égalisant les coefficients en \(1/z^5\), on trouve
\[
\Ob(\dz\wp)=-3\wp\dz\wp
\]
d'où l'on déduit
\begin{equation}\label{eq_dtaudzrho}
\dtau\dz\wp=\frac{3}{2}(5\eisen_4-\wp^2)\Eisen_1+\frac{3}{4}(-\wp+\eisen_2)\dz\wp.
\end{equation}

Par la proposition \ref{prop_ObstabJS}, \(\Ob(\Eisen_1)\in\QJS{3}{1}{1}\). On a \(\cQ_{1,1}\left(\Ob(\Eisen_1)\right)=-4\pi^2\), puis \(\cQ_{1,0}\left(\Ob(\Eisen_1)\right)=2\ic\pi\Eisen_1=\cQ_{1,0}(-\Eisen_1\eisen_2)\) et \(\cQ_{0,1}\left(\Ob(\Eisen_1)\right)=-2\ic\pi\eisen_2=\cQ_{0,1}(-\Eisen_1\eisen_2)\) ; on en déduit que \(\Ob(\Eisen_1)+\Eisen_1\eisen_2\in\JS{3}=\CC\dz\wp\). Finalement, 
\begin{equation}\label{eq_ObEun}
\Ob(\Eisen_1)=\frac{1}{2}\dz\wp-\Eisen_1\eisen_2. 
\end{equation}
Il en résulte que \(\AlgQJell\) n'est pas stable par \(\Ob\). Compte-tenu de~\eqref{eq_Euwped}, on obtient ensuite
\begin{equation}\label{eq_dtauEun}
4\dtau\Eisen_1=\Eisen_1\eisen_2+\wp\Eisen_1+\frac{1}{2}\dz\wp.
\end{equation}
De m\^eme, \(\Ob(\eisen_2)\in\QJS{4}{2}{0}\). Grâce à~\eqref{eq_Qob}, \(\cQ_{2,0}\left(\Ob(\eisen_2)\right)=4\pi^2=\cQ_{2,0}(-\eisen_2^2)\) puis \(\cQ_{1,0}\left(\Ob(\eisen_2)\right)=4\ic\pi\eisen_2=\cQ_{1,0}(-\eisen_2^2)\). On en déduit que \(\Ob(\eisen_2)+\eisen_2^2\in\JS{4}=\CC\wp^2+\CC\eisen_4\). La dépendance en \(z\) montre que \(\Ob(\eisen_2)+\eisen_2^2\in\CC\eisen_4\) puis le calcul du premier coefficient de Fourier permet de retrouver l'image de \(\eisen_2\) par la dérivation de Serre 
\begin{equation}\label{eq_ObEdeux}
 \Ob(\eisen_2) = -\eisen_2^2-5\eisen_4,
\end{equation}
et donc l'équation de Ramanujan
\begin{equation}\label{eq_ramaed}
\dtau\eisen_2=\frac{1}{4}\left(\eisen_2^2-5\eisen_4\right).
\end{equation}
\section{Crochets de Rankin-Cohen et déformations formelles}
Cette partie est consacrée à la construction de déformations formelles (voir \cite[\texten{Chapter}~13]{zbMATH06054532}, \cite[\S~1.1]{zbMATH07362171}) des différentes algèbres de formes quasi-Jacobi étudiées précédemment. 
\subsection{Crochets de Rankin-Cohen des formes quasi-Jacobi de type quasielliptique}\label{RCquasiell}
D'après la proposition \ref{prop_stabder}, la dérivation modulaire \(\dtau\) de \(\AlgQJS\) est homogène de degré \(2\) pour cette graduation : \(\dtau(\QJSpoids{k})\subseteq\QJSpoids{k+2}\) pour tout \(k\geq 0\). On peut alors définir une déformation formelle de \(\AlgQJS\) du type des crochets de Rankin-Cohen formels au sens de \cite{zbMATH07362171}.

\begin{prop}\label{RCtauQ}Considérons la suite \((\crochet{}{}{n})_{n\geq 0}\) d'applications de \(\AlgQJS\times \AlgQJS\) dans \(\AlgQJS\) définies par prolongement bilinéaire de
\begin{equation}\label{RCtau}
\crochet{f}{g}{n}=\sum_{r=0}^n(-1)^r\binom{k+n-1}{n-r}\binom{\ell+n-1}{r}\dtau^r(f)\dtau^{n-r}(g)
\end{equation}
pour tous \(f\in\QJSpoids{k},g\in\QJSpoids{\ell}\). Alors :
\begin{enumerate}[font=\normalfont,label={(\roman*)}]
\item\label{item_unprop} \([\QJSpoids{k},\QJSpoids{\ell}]_n\subseteq\QJSpoids{k+\ell+2n}\) pour tous \(n,k,\ell\geq 0\).
\item\label{item_deuxprop} La suite \((\crochet{}{}{n})_{n\geq 0}\) est une déformation formelle de \(\AlgQJS\).
\item\label{item_troisprop} La sous-algèbre \(\AlgM\) est stable par les applications \(\crochet{}{}{n}\), leur restriction coïncidant avec les crochets de Rankin-Cohen classiques sur les formes modulaires.
\end{enumerate}
\end{prop}
\begin{proof} Les points~\ref{item_unprop} et~\ref{item_deuxprop} découlent d'une application directe du résultat algébrique général de \cite[\texten{Proposition}~3]{zbMATH07362171}. Le point~\ref{item_troisprop} est le résultat classique prouvé par exemple en \cite[\S 5.2]{zbMATH05808162}.
\end{proof}
On a vu au \S~\ref{subsec_jacell} que la sous-algèbre \(\AlgQJell\) n'est pas stable par la dérivation \(\dtau\). Elle l'est cependant par la déformation ci-dessus.
\begin{thm}\label{RCtauA}
La sous-algèbre \(\AlgQJell\) est stable par la suite de crochets de Rankin-Cohen \(\left(\crochet{}{}{n}\right)_{n\geq0}\).
\end{thm} 
\begin{proof}
On utilise la méthode générale d'extension-restriction formulée au \texten{Theorem~6} de \cite{zbMATH07362171}. 
On considère pour cela l'inclusion \(A\subset R\) où l'on note \(R=\AlgQJS\) et \(A=\AlgQJell\).
On note \(\derpoids\) la dérivation de \(R\) de multiplication par la moitié du poids, c'est-à-dire définie par extension linéaire de
\begin{equation}
\derpoids(f)=\tfrac{k}{2}f\quad\text{pour tout \(f\in \QJSpoids{k}.\)}
\end{equation}
On introduit par ailleurs la dérivation de \(R\) définie par
\begin{equation}\label{defDelta}
\theta=\tfrac{1}{4}(\Ob-\Eisen_1\dz)=\dtau-\tfrac{1}{2}e_2\derpoids.
\end{equation} 
Il est clair que \(\derpoids(A) \subseteq A\). D'autre part \(A=\AlgJS[\Eisen_1]\), les dérivations \(\dz\) et \(\Ob\) stabilisent \(\AlgJS\) par la table~\vref{tab_stabderalginter}, donc \(\theta(\AlgJS)\subseteq A\) puis
\[
\theta(\Eisen_1)=\frac{1}{8}(\dz\wp+2\wp\Eisen_1)
\] 
grâce à~\eqref{eq_Euwped} et~\eqref{eq_ObEun}. On en déduit que \(\theta(A) \subseteq A\).

De plus la dérivation \(\theta\) est homogène de degré 2 pour la graduation définie par le poids sur \(R\) et on a
\begin{equation}\label{Deltatheta}
\derpoids\theta-\theta\derpoids=\theta.
\end{equation}
On pose \(x=\frac{1}{4}\eisen_2\), qui vérifie \(x\in R\) et \(x\notin A\). Il vérifie \(\derpoids(x)=x\) et~\eqref{eq_ObEdeux} montre que \(\theta(x)=-x^2-\frac{5}{16}\eisen_4\). En posant \(h=-\frac{5}{16}\eisen_4\), on a \(h\in A\) avec \(\derpoids(h)=2h\) et \(\theta(x)=-x^2+h\). 

On est donc exactement dans les conditions d'application du \texten{Theorem~6} de~\cite{zbMATH07362171} avec \(\dtau=\theta+2x\derpoids\), et l'on conclut que la suite \((\CM^{\dtau,\derpoids}_n)_{n\geq 0}\) des crochets de Connes-Moscovici associée aux deux dérivations \(\dtau\) et \(\derpoids\) définit par restriction à \(A\) une déformation formelle de \(A\). Or ces crochets ne sont autres que les crochets de Rankin-Cohen \(\left(\crochet{}{}{n}\right)_{n\geq 0}\) comme on le vérifie par un calcul combinatoire immédiat (voir par exemple la preuve de la \texten{Proposition~3} de  \cite{zbMATH07362171}).
\end{proof}
\begin{cor} La suite \((\crochet{}{}{n})_{n\geq 0}\) est une déformation formelle de \(\AlgQJell\), qui prolonge la suite des crochets de Rankin-Cohen classiques sur les formes modulaires.
\end{cor}

\begin{rem}\label{nonstabE}Les sous-algèbres \(\AlgQJmod\) et \(\AlgJS\) ne sont pas stables par les crochets \(\crochet{}{}{n}\). Par exemple, il résulte  de~\eqref{eq_dtaurho} et~\eqref{eq_ramae4} que \(\crochet{\eisen_4}{\wp}{1}\) est de profondeur \((0,1)\), il n’appartient donc à aucune de ces sous-algèbres. On construit dans ce qui suit une déformation formelle de \(\AlgJS\) qui prolonge les crochets de Rankin-Cohen classiques sur les formes modulaires. 
\end{rem}

\subsection{Crochets de Rankin-Cohen des formes de Jacobi singulières}\label{RCsing}
On commence par établir une variante de la proposition \ref{RCtauQ} en introduisant dans \(\AlgQJS\) la dérivation
\begin{equation}\label{defd}
d=\dtau +\frac{1}{4}\Eisen_1\dz=\frac{1}{4}\Ob+\frac 12 \eisen_2 \derpoids 
\end{equation}
où \(\derpoids\) est définie par la formule \eqref{defDelta}.  
\begin{prop}\label{RCdQ}Considérons la suite \((\Crochet{}{}{n})_{n\geq 0}\) d'applications de \(\AlgQJS\times\AlgQJS\) dans \(\AlgQJS\) définies par prolongement bilinéaire de
\begin{equation}\label{RCd}
\Crochet{f}{g}{n}=\sum_{r=0}^n(-1)^r\binom{k+n-1}{n-r}\binom{\ell+n-1}{r}d^r(f)d^{n-r}(g)
\end{equation}
pour tous \(f\in\QJSpoids{k},g\in\QJSpoids{\ell}\). Alors :
\begin{enumerate}[font=\normalfont,label={(\roman*)}]
\item \(\Crochet{\QJSpoids{k}}{\QJSpoids{\ell}}{n}\subset\QJSpoids{k+\ell+2n}\) pour tous \(n,k,\ell\geq 0\).
\item La suite \((\Crochet{}{}{n})_{n\geq 0}\) est une déformation formelle de \(\AlgQJS\).
\item La sous-algèbre \(\AlgM\) est stable par les applications \(\Crochet{}{}{n}\), leur restriction coïncidant avec les crochets de Rankin-Cohen classiques sur les formes modulaires.
\end{enumerate}
\end{prop}
\begin{proof} La dérivation \(d\) est homogène de degré \(2\). Il suffit donc une nouvelle fois d'appliquer la \texten{Proposition~3} de \cite{zbMATH07362171}.
\end{proof}
L'algèbre \(\AlgJS\) n'est pas stable par la dérivation \(d\), en effet elle est stable par \(\Ob\) mais ne contient pas \(\eisen_2\). Elle l'est cependant par la déformation ci-dessus.
\begin{thm}\label{RCtauE}
La sous-algèbre \(\AlgJS\) est stable par la suite de crochets de Rankin-Cohen \(\left(\Crochet{}{}{n}\right)_{n\geq0}\).
\end{thm} 
\begin{proof}
On reprend la structure de la preuve du théorème \ref{RCtauA}, avec \(A\subset R\) pour \(R=\AlgQJS\) et \(A=\AlgJS\). On introduit cette fois la dérivation de \(R\) définie par
\(\theta'=\tfrac{1}{4}\Ob\).
D'après la proposition~\ref{prop_ObstabJS}, on a \(\derpoids(A)\subset A\) and \(\theta'(A)\subset A\).

Puisque \(\theta'\) est homogène de degré 2, on a encore
\begin{equation}\label{Deltathetaprime}
\derpoids\theta'-\theta'\derpoids=\theta'.
\end{equation}
Les mêmes éléments \(x=\frac{1}{4}\eisen_2\) et \(h=-\frac{5}{16}\eisen_4\) vérifient
\[
h\in A,\ \  x\in R, \ \  x\notin A, \ \ \derpoids(x)=x,\ \ \derpoids(h)=2h, \ \  \theta'(x)=-x^2+h.
\] 
On conclut donc exactement de la même façon en appliquant le \texten{Theorem 6} de \cite{zbMATH07362171} avec cette fois \(d=\theta'+2x\derpoids\), de sorte que la suite \((\CM^{d,\derpoids}_n)_{n\geq 0}\) des crochets de Connes-Moscovici associée aux deux dérivations \(d\) et \(\derpoids\) définit par restriction à \(A\) une déformation formelle de \(A\) qui coïncide  avec la suite des crochets de Rankin-Cohen \((\Crochet{}{}{n})_{n\geq 0}\) considérée ici.
\end{proof}
\begin{cor} La suite \(\left(\Crochet{}{}{n}\right)_{n\geq0}\) est une déformation formelle de \(\AlgJS\), qui prolonge la suite des crochets de Rankin-Cohen classiques sur les formes modulaires.
\end{cor}
\begin{rem}
La construction des crochets \eqref{RCd} et la stabilité de \(\AlgJS\) sont démontrées de façon différente dans \cite[\texten{Proposition}~2.15]{zbMATH05953688}.
\end{rem}
\begin{rem}
D'après la remarque \ref{nonstabE}, la sous-algèbre \(\AlgQJmod\) n'est pas stable par \((\crochet{}{}{n})_{n\geq 0}\). En revanche, elle est trivialement stable par \((\Crochet{}{}{n})_{n\geq 0}\), puisque \(\AlgQJmod\) est stable par \(\Ob\). On montre que \(\Crochet{\Eisen_1}{\eisen_4}{1}\) est de profondeur modulaire \(1\) (par exemple en utilisant~\eqref{eq_Obeq} et~\eqref{eq_ObEun}) de sorte que \(\AlgQJell\) n'est pas stable par \(\left(\Crochet{}{}{n}\right)_{n\geq0}\). 
\end{rem}
\begin{rem}
 La construction des crochets de Rankin-Cohen aux propositions \ref{RCtauQ} et \ref{RCdQ} repose sur les relations \eqref{Deltatheta} et \eqref{Deltathetaprime} vérifiées pour les dérivations utilisées. Une construction très différente de déformation formelle de l'algèbre \(\AlgQJS\) est proposée dans ce qui suit, en utilisant les dérivations \(\dtau\) et \(\dz\), qui vérifient \(\dtau \circ \dz = \dz \circ \dtau\).
\end{rem}

\subsection{Transvectants des formes quasi-Jacobi de type quasimodulaire}\label{subsec_transvec}
\begin{prop}\label{TransQ} 
Considérons la suite \((\accol{}{}{n})_{n\geq 0}\) d'applications bilinéaires de \(\AlgQJS\times\AlgQJS\) dans \(\AlgQJS\) définies par 
\begin{equation}\label{eq_TransQ}
\accol{f}{g}{n}=\sum_{r=0}^n{(-1)^r}\binom{n}{r}\dtau^{n-r}\dz^{r}(f)\dtau^{r}\dz^{n-r}(g)\quad f,g\in \AlgQJS  
\end{equation}
\begin{enumerate}[font=\normalfont,label={(\roman*)}]
\item\label{item_schubert} La suite \((\dfrac1{n!}\accol{}{}{n})_{n\geq 0}\) est une déformation formelle de \(\AlgQJS\).
\item\label{item_beethoven} \(\accol{\QJSpoids{k}}{\QJSpoids{\ell}}{n}\subset \QJSpoids{k+\ell+3n}\) pour tous \(n,k,\ell\geq 0\).
\end{enumerate}
\end{prop}
\begin{proof}
Le point~\ref{item_schubert} est un résultat classique de théorie des invariants correspondant à l'associativité du produit de Moyal (voir par exemple \cite[\texten{Proposition}~5.20]{zbMATH01516969}). 
Le point~\ref{item_beethoven} découle du fait que \(\dtau\) et \(\dz\) sont homogènes de degrés respectifs 2 et 1. 
\end{proof}
\begin{rem}
On rappelle les deux propriétés suivantes générales des transvectants utilisées dans la suite. D'une part ils satisfont la relation de recurrence :
\begin{equation}\label{rectrans}
\accol{f}{g}{n+1}=\accol{\dtau f}{\dz g}{n}-\accol{\dz f}{\dtau g}{n}
\end{equation}
initialisée par le fait que \(\accol{}{}{0}\) est le produit dans \(\AlgQJS\times\AlgQJS\), et \(\accol{}{}{1}\) est le crochet de Poisson \(\dtau\wedge\dz\)~:
\[
\accol{f}{g}{0}=fg \qquad \text{ et } \qquad \accol{f}{g}{1}=\dtau(f)\dz(g)-\dz(f)\dtau(g). 
\]
D'autre part l'associativité du star-produit défini sur \(\AlgQJS[[\planck]]\) à partir de
\begin{equation}\label{star}
\forall(f,g)\in\AlgQJS\times\AlgQJS\qquad f\star g=\sum_{n\geq 0}\dfrac 1{n!}\{f,g\}_n\planck^n
\end{equation} 
équivaut à : 
\begin{equation}\label{asstrans}
\forall(f,g,h)\in\AlgQJS\times\AlgQJS\times\AlgQJS\qquad\sum_{r=0}^n\binom nr\accol*{\accol{f}{g}{r}}{h}{n-r}=\sum_{r=0}^n\binom nr\accol*{f}{\accol{g}{h}{r}}{n-r}. 
\end{equation}
\end{rem}
On a vu \S~\ref{sec_fqjm} que  \(\AlgQJmod\) est stable par \(\dz\) mais pas par \(\dtau\). Elle l'est cependant par les transvectants comme on va le voir ci-dessous. La preuve nécessite quelques résultats techniques préliminaires.
\begin{lem}\label{lemmeA}On considère la dérivation \(d=\dtau +\tfrac{1}{4}\Eisen_1\dz\) de \(\AlgQJS\) ; on a :
\begin{enumerate}[font=\normalfont,label={(\roman*)}]
\item \(d(f)\in\AlgQJmod\) et \(\{f, g\}_1 \in \AlgQJmod\) pour tous \(f,g\in\AlgQJmod\) ;
\item \(d(E_1)\in\AlgQJmod\) et \(\{f, E_1\}_1 \in \AlgQJmod\) pour tout \(f\in\AlgQJmod\).
\end{enumerate}
\end{lem}
\begin{proof}
On a déjà considéré en \eqref{defd} la dérivation \(d=\dfrac 14 \Ob + \dfrac 12 \eisen_2 \derpoids\). L'algèbre \(\AlgQJmod=\AlgJS[\eisen_2]\) est stable par \(\Ob\) d'après le \S~\ref{subsubsec_application} et elle est donc stable par \(d\). On calcule pour tous \(f,g\in \AlgQJmod\) :
\[ 
\accol{f}{g}{1}=\partial_{\tau}(f)\dz(g)-\dz(f)\partial_{\tau}(g)=d(f)\dz(g)-\dz(f)d(g)\in\AlgQJmod
\]
puisque \(\AlgQJmod\) est stable par \(d\) et par \(\dz\) d'après le \S~\ref{sec_fqjm}. 

Il résulte de~\eqref{eq_ObEun} que
 \begin{equation}\label{eq_dEun}
 d(\Eisen_1)=\tfrac{1}{8}\dz\wp\in\AlgJS\subseteq\AlgQJmod.
\end{equation}
Enfin, grâce à \eqref{eq_Euwped} : 
\[
\accol{f}{\Eisen_1}{1}=d(f)\dz(\Eisen_1)-d(\Eisen_1)\dz(f)=-(\wp+\eisen_2)d(f)-\frac{1}{8}\dz(f)\dz\wp\in \AlgQJmod. 
\]
\end{proof}
\begin{rem}
Pour tout \(n\in\N\), on a \(d(\Eisen_1^n)=\tfrac{n}{8}(\dz\wp)\Eisen_1^{n-1}\in\AlgQJell=\AlgJS[\Eisen_1]\). Mais \(\AlgQJell\) n'est pas stable par \(d\) puisque par exemple \(d\wp=\tfrac{1}{4}\Ob(\wp)+\tfrac{1}{2}\wp\eisen_2\) avec \(\Ob(\wp)\in\AlgJS\) (voir proposition~\ref{prop_ObstabJS}) et \(\wp\eisen_2\notin\AlgQJell\). 
\end{rem}
\begin{lem}\label{lemmeB}
Soit \(n \geq 1\) un entier vérifiant les deux propriétés suivantes : 
\begin{enumerate}[font=\normalfont,label={(H\arabic*)}]
 \item\label{item_hypun} pour tous \(f, g \in \AlgQJmod\), on a  \(\accol{f}{g}{n}\in \AlgQJmod\) ;
 \item\label{item_hypdeux} pour tous \(f, g \in \AlgQJmod\), on a \(\accol{f\Eisen_1}{g}{n}-\accol{f}{g\Eisen_1}{n}\in\AlgQJmod\). 
 \end{enumerate}
 Alors, pour tous \(f, g \in \AlgQJmod\), on a \(\accol{f}{g}{n+1} \in\AlgQJmod\) et \(\accol{f}{\Eisen_1}{n+1}\in\AlgQJmod\). 
\end{lem}
\begin{proof}D'après la formule de récurrence \eqref{rectrans}, on a 
\begin{align*}
\accol{f}{g}{n+1} &=\accol{\dtau f}{\dz g}{n}-\accol{\dz f}{\dtau g}{n}\\
&=-\tfrac{1}{4}\left(\accol{\dz(f)\Eisen_1}{\dz(g)}{n}-\accol{\dz(f)}{\dz(g)\Eisen_1}{n}\right)+\left(\accol{d(f)}{\dz(g)}{n}- \accol{\dz(f)}{d(g)}{n}\right).
\end{align*}
Or \(\accol{\dz(f)\Eisen_1}{\dz(g)}{n}-\accol{\dz(f)}{\dz(g)\Eisen_1}{n}\in\AlgQJmod\) d'après l'hypothèse~\ref{item_hypdeux} appliquée aux éléments \(\dz(f)\) et \(\dz(g)\) de \(\AlgQJmod\). De même comme \(d(f)\) et \(d(g)\) appartiennent à \(\AlgQJmod\) d'après le lemme \ref{lemmeA}, la différence \(\{d(f), \dz(g)\}_n- \{\dz(f), d(g)\}_n\) est aussi un élément de \(\AlgQJmod\) par l'hypothèse~\ref{item_hypun} . On conclut que \(\accol{f}{g}{n+1}\in\AlgQJmod\).
La même démonstration s'applique à \(f \in \AlgQJmod\) et \(g =\Eisen_1\) puisque \(\dz(\Eisen_1)\) et \(d(\Eisen_1)\) sont des éléments de \(\AlgQJmod\) d'après~\eqref{eq_Euwped} et \eqref{eq_dEun}. On a donc \(\accol{f}{\Eisen_1}{n+1}\in\AlgQJmod\), ce qui achève la preuve.
\end{proof}

\begin{lem}\label{lemmeC}
 Pour tout \(n \geq 1\) et toutes \(f,g\in\AlgQJS\), on a :
 \begin{multline*}
 \accol{f\Eisen_1}{g}{n}-\accol{f}{g\Eisen_1}{n}=f\accol{\Eisen_1}{g}{n} +(-1)^{n-1} g\accol{\Eisen_1}{f}{n}\\-\sum_{i=1}^{n-1}
\binom{n}{i} \left(
 \accol*{\accol{f}{\Eisen_1}{i}}{g}{n-i}+(-1)^{n-1}\accol*{\accol{g}{\Eisen_1}{i}}{f}{n-i}
 \right). 
 \end{multline*}
\end{lem}
\begin{proof}
D'une part on peut réécrire chaque produit comme un crochet \(\accol{}{}{0}\), d'autre part pour tout \(0\leq j\leq n\) le crochet \(\accol{}{}{j}\) est \((-1)^j\)-symétrique. L'égalité voulue peut donc être reformulée sous la forme
\begin{multline*}
\accol*{\accol{f}{\Eisen_1}{0}}{g}{n}-\accol*{f}{\accol{\Eisen_1}{g}{0}}{n}=\accol*{f}{\accol{\Eisen_1}{g}{n}}{0}-\accol*{\accol{f}{\Eisen_1}{n}}{g}{0}\\
 -\sum_{i=1}^{n-1}\binom ni\accol*{\accol{f}{\Eisen_1}{i}}{g}{n-i}+\sum_{i=1}^{n-1}\binom ni\accol*{f}{\accol{\Eisen_1}{g}{i}}{n-i}
 \end{multline*}
 c'est-à-dire
\[
\sum_{i=0}^n\binom ni\accol*{\accol{f}{\Eisen_1}{i}}{g}{n-i}=\sum_{i=0}^{n}\binom ni\accol*{f}{\accol{\Eisen_1}{g}{i}}{n-i}. 
\]
D'après \eqref{star} et \eqref{asstrans} cette identité traduit l'égalité \((f\star\Eisen_1)\star g=f\star(\Eisen_1\star g)\). Cette dernière égalité est satisfaite pour tous \(f\) et \(g\) dans \(\AlgQJS\) grâce au point~\ref{item_schubert} de la proposition~\ref{TransQ}. 
\end{proof}

\begin{lem}\label{lemmeD}
On a \(\accol{f}{g}{n}\in\AlgQJmod\) et \(\accol{f}{\Eisen_1}{n}\in\AlgQJmod\) pour tout \(n \geq 1\) et tous \(f,g\in \AlgQJmod\).
\end{lem}

\begin{proof} 
On raisonne par récurrence sur \(n\). Le cas \(n=1\) est montré au lemme~\ref{lemmeA}. Si la propriété est vraie pour tout \(1 \leq i \leq n\),  le lemme \ref{lemmeC} montre alors que pour tous \(f, g\in \AlgQJmod\) on a \(\accol{f\Eisen_1}{g}{n}-\accol{f}{g\Eisen_1}{n}\in\AlgQJmod\).  On conclut avec le lemme \ref{lemmeB} que,  \(\accol{f}{g}{n+1}\in\AlgQJmod\) et \(\accol{f}{\Eisen_1}{n+1}\in\AlgQJmod\) pour tous \(f, g \in \AlgQJmod\). 
 \end{proof}

On a ainsi démontré que :
\begin{thm}\label{TransB}
La suite \((\dfrac 1{n!}\accol{}{}{n})_{n\geq 0}\) est une déformation formelle de \(\AlgQJmod\).
\end{thm} 
\begin{proof}
Résulte immédiatement du lemme ci-dessus et du point~\ref{item_schubert} de la proposition \ref{TransQ}.
\end{proof}

\begin{rem}
Les sous-algèbres \(\AlgQJell\) et \(\AlgJS\) ne sont pas stables par \(\left(\accol{}{}{n}\right)_{n\geq0}\) puisque par exemple \(\{\eisen_4,\wp\}_1\notin\AlgQJell\) d'après \eqref{eq_ramae4sanse6}. Les crochets \(\accol{}{}{n}\) s'annulent sur \(\mathcal{\AlgM}\) pour tout \(n\geq 1\). La structure de Poisson sur \(\AlgQJmod\) définie par le crochet \(\accol{}{}{1}\)  est étudiée dans \cite{zhou}. On résume la situation page~\pageref{fig_recap}
\end{rem}

\begin{rem}
 Avec le point~\ref{item_beethoven} de la proposition \ref{TransQ}, le théorème \ref{TransB} permet de construire, à partir de deux formes quasi-Jacobi de type quasimodulaire de poids respectifs \(k\) et \(\ell\), une nouvelle forme dans \(\AlgQJmod\) de poids \(k+\ell+3n\), pour tout \(n \geq 0\). C'est un processus comparable à celui obtenu aux sections \ref{RCquasiell} et \ref{RCsing} avec les crochets de Rankin-Cohen sur les formes quasi-Jacobi de type quasielliptique et sur les formes elliptiques, l'augmentation du poids étant dans ces cas de \(2n\).
\end{rem}

\Ajout{
\begin{rem}
L'objectif originel de cette étude est la construction de déformations formelles sur une algèbre contenant les fonctions elliptiques obtenues à partir de la fonction de Weierstra\ss{} et sa dérivée. C'est ce qui nous a motivé à introduire l'algèbre \(\CC[\wp,\dz\wp,\eisen_4]\) puis \(\CC[\wp,\dz\wp,\eisen_4,\Eisen_1,\eisen_2]\)  pour gagner la stabilité par dérivation. Il est alors naturel de considérer les algèbres que nous avons appelées \emph{de type quasimodulaire} et \emph{de type elliptique}. Un changement de contexte consiste à prendre pour point de départ la notion de forme de Jacobi d'indice non nécessairement nul, éventuellement sur un sous-groupe du groupe de Jacobi puis de chercher à construire des déformations formelles dans ce contexte. Ce travail reste à faire. 
\end{rem}
}

\vspace{\fill}
\pagebreak[0]
\vspace{-\fill}

\appendix

\section{Stabilité des différentes algèbres par les différents crochets}

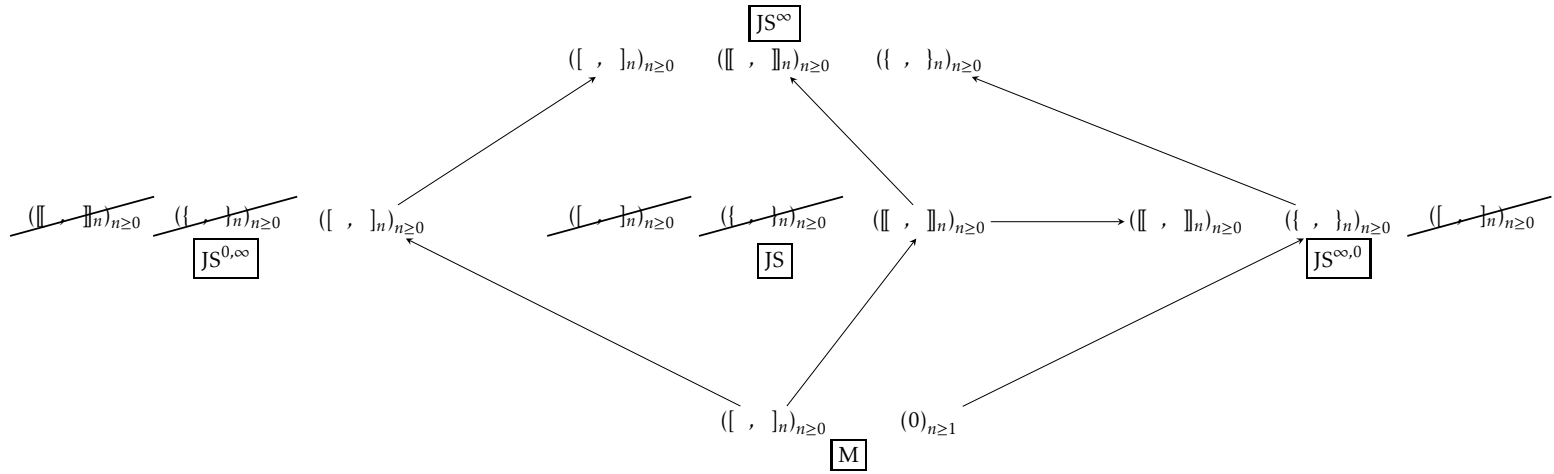
\begin{tikzpicture}
\node[rotate=-90] at (0,0) {%
\begin{minipage}{0.8\textheight}
\scalebox{0.84}{
\begin{tabularx}{\textwidth}{CCCCCCCCCCC}
&&&& \multicolumn{3}{c}{\(\boxed{\AlgQJS}\)} &&&&\\
&&&& \rose{\tikzmark{a}{\(\left(\crochet{}{}{n}\right)_{n\geq0}\)}} & \rose{\tikzmark{c}{\(\left(\Crochet{}{}{n}\right)_{n\geq0}\)}}& \bleu{\tikzmark{i}{\(\left(\accol{}{}{n}\right)_{n\geq0}\)}} &&&&\\
&&&&&&&&&&\\[10ex]
\rose{\(\left(\Crochet{}{}{n}\right)_{n\geq0}\)}  & \bleu{\(\left(\accol{}{}{n}\right)_{n\geq0}\)} & \rose{\tikzmark{b}{\(\left(\crochet{}{}{n}\right)_{n\geq0}\)}} & \phantom{xxxxxx} & \rose{\(\left(\crochet{}{}{n}\right)_{n\geq0}\)} & \bleu{\(\left(\accol{}{}{n}\right)_{n\geq0}\)} & \rose{\tikzmark{d}{\(\left(\Crochet{}{}{n}\right)_{n\geq0}\)}} & \phantom{xxxxxx} & \rose{\tikzmark{e}{\(\left(\Crochet{}{}{n}\right)_{n\geq0}\)}} & \bleu{\tikzmark{h}{\(\left(\accol{}{}{n}\right)_{n\geq0}\)}} & \rose{\(\left(\crochet{}{}{n}\right)_{n\geq0}\)}\\
\multicolumn{3}{c}{\(\boxed{\AlgQJell}\)} && \multicolumn{3}{c}{\(\boxed{\AlgJS}\)} && \multicolumn{3}{c}{\(\quad\boxed{\AlgQJmod}\)} \\
&&&&&&&&&&\\[10ex]
&&&&& \rose{\tikzmark{g}{\(\left(\crochet{}{}{n}\right)_{n\geq0}\)}} & \bleu{\tikzmark{j}{\(\left(0\right)_{n\geq 1}\)}} &&&&\\
&&&&&  \multicolumn{2}{c}{\(\boxed{\AlgM}\)} &&&&
\end{tabularx}
\MeasureLastTable{11} 
\rose{\CrossOut{31}}
\bleu{\CrossOut{32}}
\rose{\CrossOut{35}}
\bleu{\CrossOut{36}}
\rose{\CrossOut{41}}
\label{fig_recap}
\rose{\link{b}{a}}
\rose{\link{g}{b}}
\rose{\link{g}{d}}
\rose{\link{d}{c}}
\rose{\link{d}{e}}
\bleu{\link{h}{i}}
\bleu{\link{j}{h}}
}
\vspace*{10ex}

Les flèches indiquent des prolongements ; lorsqu'un crochet est barré, cela signifie qu'il ne stabilise pas l'algèbre.
\end{minipage}
};
\end{tikzpicture}

\FloatBarrier

\newpage

\section{Dimensions des sous espaces des formes quasi-Jacobi d'indice nul}
Si \(k\) est un entier, on définit \(\pa(k)=\frac{1+(-1)^k}{2}\) et \(\ia(k)=\frac{1-(-1)^k}{2}\).
\begin{thm}\label{thm_dimsousev}
Soit \(k\geq 0\) un entier. Les dimensions \(\fds(k)\) de \(\JS{k}\), \(\fdell(k)\) de \(\JSell{k}\), \(\fdmod(k)\) de \(\JSmod{k}\) et \(\fdq(k)\) de \(\QJSpoids{k}\) sont données par
\begin{align*}
\fds(k) &= \frac{107}{288}+\frac{3}{16}k+\frac{1}{48}k^2+\frac{9}{32}(-1)^k+\frac{1}{16}(-1)^kk+\frac{1}{8}\left(\pa(k)+\ia(k)\ic\right)\ic^k+\frac{1}{9}\left(\jc^k+\jc^{2k}\right)\\
\fdell(k) &= \frac{175}{288}+\frac{15}{32}k+\frac{5}{48}k^2+\frac{1}{144}k^3+\frac{5}{32}(-1)^k+\frac{1}{32}(-1)^kk+\frac{1}{8}\pa(k)\ic^k+\frac{1}{27}(1-\jc)\jc^k +\frac{1}{27}(2+\jc)\jc^{2k}\\
\fdmod(k) &=\begin{multlined}[t]\frac{121}{288}+\frac{55}{192}k+\frac{11}{192}k^2+\frac{1}{288}k^3+\frac{13}{32}(-1)^k+\frac{11}{64}(-1)^kk+\frac{1}{64}(-1)^kk^2+\frac{1}{16}\left(\pa(k)+\ia(k)\ic\right)\ic^k\\+\frac{1}{27}(2+\jc)\jc^k +\frac{1}{27}(1-\jc)\jc^{2k}\end{multlined}\\
\fdq(k) &=\begin{multlined}[t]\frac{4267}{6912}+\frac{55}{96}k+\frac{199}{1152}k^2+\frac{1}{48}k^3+\frac{1}{1152}k^4+\frac{ 63}{256}(-1)^k+\frac{3}{32}(-1)^kk+\frac{1}{128}(-1)^kk^2\\
+\frac{1}{16}\pa(k)\ic^k+\frac{1}{27}\left(\jc^k+\jc^{2k}\right)
\end{multlined}
\end{align*}
où \(\jc=\exp(2\ic\pi/3)\). 
\end{thm}
\begin{proof}
Avec le même argument qu'en la proposition~\ref{prop_dimJS}, les séries génératrices des dimensions sont
\begin{align*}
\sum_{k\in\N}\fds(k)\cdot z^k&=\frac{1}{(1-z^2)(1-z^3)(1-z^4)},\\
\sum_{k\in\N}\fdell(k)\cdot z^k&=\frac{1}{(1-z)(1-z^2)(1-z^3)(1-z^4)},\\
\sum_{k\in\N}\fdmod(k)\cdot z^k&=\frac{1}{(1-z^2)^2(1-z^3)(1-z^4)}
\shortintertext{et}
\sum_{k\in\N}\fdq(k)\cdot z^k&=\frac{1}{(1-z)(1-z^2)^2(1-z^3)(1-z^4)}. 
\end{align*}
La décomposition en éléments simples des membres de droites permet de justifier que les dimensions sont de la forme
\[
P_1(k)+P_{-1}(k)(-1)^k+P_{\ic}(k)\ic^k+P_{-\ic}(k)(-\ic)^k+P_{\jc}(k)\jc^k+P_{\jc^2}(k)\jc^{2k}
\]
où les \(P_\xi\) sont des polynômes de degré strictement majoré par la valuation de \(z-\xi\) dans le dénominateur de la fonction génératrice (voir par exemple~\cite[\texten{Theorem}~ 4.4.1]{zbMATH06016068}). Ces polynômes sont facilement déterminés par le début du développement des séries génératrices. On a utilisé \texttt{PARI/GP} pour nos calculs~\cite{PARI2}. 
\end{proof}

À partir des formules du théorème~\ref{thm_dimsousev}, on peut avoir des formules polynomiales à coefficients rationnels dans chaque classe du poids modulo \(12\). De telles formules permettent d'obtenir des expressions \og compactes~\fg{} des dimensions semblables à l'égalité~\eqref{eq_dimexplicit} de la proposition~\ref{prop_dimJS}, par exemple
\[
\fds^{0,\infty}(k)=\projZ*{
\dfrac{1}{144}
\left(k^3+15k^2+\begin{cases}72k+144 & \text{si \(k\) est pair}\\ 63k+65 & \text{sinon}\end{cases}\right)
}.
\]
Une telle formule est cependant assez artificielle, notamment parce qu'elle n'est pas unique dans sa forme. On a aussi par exemple
\[
\fds^{0,\infty}(k)=\projZ*{
\frac{k+3}{144}\begin{cases}(k+6)^2 & \text{si \(k\) est pair}\\ (k+3)(k+9) & \text{sinon}\end{cases}
}.
\]


\bibliographystyle{plain-fr}
\bibliography{2024_DuMaRo}
\end{document}